\documentclass[journal,twoside,web]{ieeecolor}
\usepackage{generic}
\usepackage{multirow}
\usepackage{cite}
\usepackage{amsmath,amssymb,amsfonts}
\usepackage{graphicx}
\usepackage{algorithm,algorithmic}
\usepackage{hyperref}
\hypersetup{hidelinks=true}
\usepackage{textcomp}
\def\BibTeX{{\rm B\kern-.05em{\sc i\kern-.025em b}\kern-.08em
		T\kern-.1667em\lower.7ex\hbox{E}\kern-.125emX}}
\usepackage{preamble_papers}
\usepackage{mathbbol}
\usepackage{comment}
\usepackage[T1]{fontenc}
\usepackage{setspace}
\renewcommand{\P}{\mathbb{P}}

\renewcommand{\d}{\mathbb{d}}

\renewcommand{\ss}{\mathbf{s}}

\begin{document}
	
	
	\title{Uncertainty Partitioning with Probabilistic Feasibility and Performance Guarantees for Chance-Constrained Optimization}
	\author{F. Cordiano, M. Jafarian, \IEEEmembership{Member, IEEE}, and B. De Schutter, \IEEEmembership{Fellow, IEEE}
		\thanks{This project has received funding from the European Research Council (ERC) under the European Union's Horizon 2020 research and innovation programme (Grant agreement No. 101018826 - ERC Advanced Grant CLariNet).}
		\thanks{All authors are affiliated with the Delft Center for Systems and Control, Delft University of Technology, Delft, The Netherlands, email: $\{$f.cordiano, m.jafarian, b.deschutter$\}$@tudelft.nl }
	}

	\maketitle

	\begin{abstract}
		We propose a novel distribution-free scheme to solve optimization problems where the goal is to minimize the expected value of a cost function subject to probabilistic constraints. Unlike standard sampling-based methods, our idea consists of partitioning the uncertainty domain in a user-defined number of sets, enabling more flexibility in the trade-off between conservatism and computational complexity. We provide sufficient conditions to ensure that our approximated problem is feasible for the original stochastic program, in terms of chance constraint satisfaction. In addition, we perform a rigorous performance analysis, by quantifying the distance between the optimal values of the original and the approximated problem. We show that our approach is tractable for optimization problems that include model predictive control of piecewise affine systems, and we demonstrate the benefits of our approach, in terms of the trade-off between conservatism and computational complexity, on a numerical example.
	\end{abstract}


	\section{Introduction}
	Optimization problems with uncertainty are of central aspect in several fields, ranging from control \cite{mesbah2016stochastic}, to operation research \cite{dariano2019integrated}, and mathematical finance \cite{bemporad2010scenariobased}. For example, in the field of optimal control, most problems include some level of uncertainty, originating, e.g., from model uncertainty, or exogenous stochastic signals. Under these circumstances, neglecting uncertain components in the underlying optimization problem may significantly degrade the quality of the solution, in terms of performance, safety, and reliability. 
	
	Among the diverse paradigms proposed in the literature, stochastic optimization \cite{ruszczynski2003stochastic} has recently gained a lot of interest, thanks to the possibility of optimizing an expected cost while guaranteeing constraint satisfaction with a desired probability level, leading to solutions that are less conservative compared to purely robust approaches \cite{bertsimas2022robust}. On the other hand, dealing with probabilistic constraints or expected cost functions is difficult: First, it generally leads to non-convex, or even intractable, optimization problems, even for simple cost and constraint functions \cite{luedtke2010integer}; Second, it requires a probabilistic description of the uncertainty; Third, even when the two previous points are satisfied, e.g., by suitable assumptions or approximations, the solved problem should deliver guarantees in terms of performance and probability of constraint violations, while not resulting in excessively conservative solutions.
	
	These aspects have constantly motivated significant research efforts within this field. For example, several approaches assume knowledge of the uncertainty distribution \cite{weissel2009stochastic, bernardini2009scenariobased, cordiano2024scenario}, or of the statistical moments \cite{cordiano2023provablystable}. Although these approaches provide tractable approximations of the original problem and feasibility guarantees, assumptions on the specific type of the distribution are often not met in practice. For this reason, sampling-based approaches gained a lot of success in the recent years. Initially, a sample-average approximation for stochastic optimization was proposed in \cite{kleywegt2002sample}, and feasibility guarantees for chance-constrained problems were given in \cite{luedtke2008sample}. However, a common drawback of this approach is that the computational complexity grows rapidly in the number of samples, and typically many samples are required to provide feasible solutions. In this regard, randomized approaches \cite{campi2008exact}, \cite{alamo2009randomized} replace a chance constraint with a number of independent hard constraints, offering an elegant theory to control the risk of constraint violation while leading to more tractable schemes compared to the sample-average approximation. However, a common feature of randomized approaches is that the sample complexity depends on the number of decision variables; hence, they may result in high computational complexity and conservatism \cite{prandini2012randomized}, especially for high-dimensional or non-convex settings \cite{mohajerinesfahani2015performance}. 

	As most of the papers in the literature focus on feasibility, less attention has been devoted to performance guarantees, i.e., the distance between the optimal value of the original problem and the approximated one. Continuity results for the optimal value of a stochastic optimization problem are proposed in \cite{romisch2007stability}, where the uncertainty only appears in the cost function, and in \cite{romisch2003stability, henrion2004holder}, which also consider chance constraints. These works focus on relating the effect of a perturbation, in cost or constraints, to the optimal value of the problem, which is the case, e.g., in scenario reduction approaches \cite{dupavcova2003scenario}. However, such bounds require constants that might be difficult to compute explicitly, and may be often conservative, especially in the chance-constrained case. 
	In contrast, the work \cite{luedtke2008sample} proposes an approach to practically generate a lower bound to the optimal value of a chance-constrained problem by solving an auxiliary optimization problem based on a sample-average approximation, which, however, may not be computationally efficient to solve. In the context of randomized algorithms, the sampling-and-discarding approach proposed in \cite{campi2011samplinganddiscarding} provides a significant contribution in this direction, but explicit performance bounds are missing. In \cite{mohajerinesfahani2015performance}, explicit performance bounds are given for randomized solutions of chance-constrained programs with deterministic cost functions, but the resulting sample complexity scales exponentially with the dimension of the uncertainty, and explicit computation requires knowledge of the distribution. 
	
	From a practical viewpoint, both feasibility and performance guarantees are relevant in several applications. For example, in stochastic optimal control problems,
	probabilistic constraint satisfaction is crucial for the  safety of the system against uncertainty \cite{mesbah2016stochastic, farina2016stochastic}, while optimizing expected performance has important consequences on the stability of the perturbed system (where stability is indeed intended in a stochastic sense, e.g.\ as in \cite{mcallister2023nonlinear}). However, in stochastic model predictive control, the majority of the results are for linear systems \cite{prandini2012randomized}, \cite{farina2016stochastic}, for which computational tractability of the resulting optimization problem is often possible.
	
	In view of these considerations, a tractable approximation for chance-constrained optimization problems, with expected performance minimization, that is computationally well-tractable and not too conservative, and endowed with non-conservative feasibility and performance guarantees, is missing. To tackle this issue, in this paper we propose a novel distribution-free method to solve such stochastic optimization problem, which lies in between a sampling-based method and a robust approach. We provide a lower bound on the number of samples for probabilistic feasibility guarantees, and we derive informative performance bounds. The novel contributions of this paper are summarized as follows:

	\subsubsection{Trading off complexity and conservatism}
	We partition the uncertainty domain into $K$ sets, where $K$ is user-defined and can be larger than one, as opposed to \cite{margellos2014road, shang2019datadriven}. Then, a set of samples is used in an offline phase to establish a uniform concentration bound between the true probability that the uncertainty belongs to one of these sets, and the empirical one. Notably, the required sample complexity in our framework will depend on $K$, rather than on the complexity of the optimization problem as it is in \cite{luedtke2008sample}, \cite{campi2008exact}. Also, our partition-based approach is applicable to a general chance-constrained optimization problem, in contrast to \cite{sartipizadeh2019voronoi}, which proposes a scenario reduction approach, based on a Voronoi partition, specifically for stochastic reachability problems (i.e., where the probability of staying in a safe set is maximized).
	
	\subsubsection{A priori feasibility} Based on this uncertainty partition, we provide a tractable substitute for the original problem that replaces the chance constraint with a mixed-integer-robust constraint.
	Accounting for the approximation error in the previous point, we establish a sufficient condition that guarantees feasibility \emph{a priori} for the original chance-constrained problem. Since the samples are used only in the partitioning phase, the computational complexity of the resulting problem will not depend on the number of samples (as is the case for \cite{luedtke2008sample}, \cite{campi2008exact}), but rather on $K$. Thus, unlike standard sampling-based methods, our approach offers greater flexibility thanks to the choice of $K$, enabling a trade-off between computational and sample complexity with the conservatism of the performance.
	
	\subsubsection{Explicit performance bounds} We provide an interval in which the optimal value of the original chance-constrained problem is contained. To this end, we introduce an auxiliary optimization problem, which, in contrast to the approach in \cite{luedtke2008sample}, enjoys superior computational tractability. The interval we obtain is probabilistic, and we explicitly characterize the validity of this probabilistic bound in terms of the available data. 
	Compared to \cite{kleywegt2002sample, luedtke2008sample, romisch2007stability}, the bounds we derive are explicit, efficient to compute, and with an acceptable degree of conservatism for a wide class of problems. In addition, we derive an analytic bound on the distance between the optimal values of the original and approximated problem, for a class of constraint functions. Unlike \cite{romisch2003stability, henrion2004holder}, our analytic bound links the optimality gap to purely geometric properties of the cost and constraint functions, thus avoiding constants that might be difficult to compute.
	
	\subsubsection{MPC with logical constraints} We show that the proposed method is applicable to several control problems of interest, such as chance-constrained optimal control of linear systems with logical constraints, e.g., involving discrete decision variables, and chance-constrained optimal control of piecewise-affine (PWA) systems \cite{bemporad1999control}. In particular, the latter has received attention only in \cite{vignali2018datadriven} and \cite{margellos2016constrained} via randomization, and in \cite{xu2019model} via moment-based inequalities. We also propose a dedicated uncertainty partitioning strategy tailored for model predictive control (MPC) applications.
	
	The rest of the paper is organized as follows: Section \ref{sec:prob} introduces the problem formulation. In Section \ref{sec:partition}, we propose our scheme based on uncertainty partitioning. Section \ref{sec:main_section} provides feasibility and performance bounds, while in Section \ref{sec:discussions} we discuss the results from a computational viewpoint and provide details for practical implementation in optimal control problems. Last, in Section \ref{sec:example} we simulate our approach on a model predictive control problem for PWA systems, and Section \ref{sec:conc} concludes the paper.

	\section{Preliminaries and problem formulation}\label{sec:prob}

	\subsection{Mathematical notation}
	With $\Z_{>0}, \R_{>0}$, we denote, respectively, the set of the positive integers and the set of the positive real numbers (and analogous definitions hold for $\Z_{\geq0}, \R_{\geq0}$). The symbol $\mathbf{1}_{[X]}$ denotes the indicator function, which takes the value $1$ if the logical statement $X$ is true and 0 otherwise. The symbol $\lor$ denotes the logical inclusive "or", whereas $\oplus$ denotes the Minkowski sum, i.e.\ $\Xcal\oplus\Ycal=\{z: z=x+y, x\in\Xcal, y
	\in\Ycal\}$. The symbol $\B$ denotes the unit ball with center in the origin.
	A probability space is denoted by $(\Xi,\mathfrak{B}(\Xi), \P)$, where $\Xi$ is a metric space,  $\mathfrak{B}(\Xi)$ is its Borel $\sigma$-algebra, and $\P:\mathfrak{B}(\Xi)\rightarrow[0,1]$ is a probability measure. The product measure $\P^N$ is defined on $\Xi^N=\Xi\times...\times\Xi$, and represents the joint distribution of $(X_1,...,X_N)$, when all $X_i$ are independently and identically distributed according to $\P$.

	\subsection{Problem formulation}
	Consider $\Xcal\subset\R^{n_x}$,  $\Theta\subset\R^{n_\theta}$, and the probability space $(\Theta,\mathfrak{B}(\Theta), \P)$. Given the measurable functions $J: \Xcal\times\Theta\rightarrow\R_{\geq0}$ and $g: \Xcal\times\Theta\rightarrow\R$, let us consider the following chance-constrained optimization problem:
	\begin{align*}
		\text{CP}_\veps: \quad \min_{x\in\Xcal} \ & \E [J(x, \theta)]
		\\ \text{s.t.} \ & \Pr\left(g(x, \theta) \leq 0 \right) \geq 1-\veps, 
	\end{align*}
	with optimal value $J^\star_{\text{CP}_\veps}$ and feasible set $\Fcal_{\text{CP}_\veps}$. Also, $\theta\in\R^{n_\theta}$ is stochastic uncertainty, and $\veps\in[0,1)$ is the risk parameter.
	
	Let $Z\in\Z_{>0}$. In this article, we consider a class of constraint functions described by
	\begin{align}\label{eq:g}
		g(x, \theta) := \min_{h\in\{1,...,Z\}} g_h(x, \theta)
	\end{align}
	where $g_h: \Xcal\times\Theta\rightarrow\R, \forall h\in\{1,..., Z\}$ are measurable functions. Together with \eqref{eq:g}, the feasible set of CP$_\veps$ can be equivalently written as
	\begin{align}\label{eq:cc}
		\Fcal_{\text{CP}_\veps} = \{x\in\Xcal: \Pr\left(\lor_{h=1}^Z (g_h(x, \theta) \leq 0) \right) \geq 1-\veps\},
	\end{align}
	that is, we require that the probability that at least one of the functions $g_h$ is non-positive is at least $1-\veps$. Note that this formulation is similar to the one considered in \cite{mohajerinesfahani2015performance}, and includes in particular chance constraints with integer variables. 
	Also, as we will discuss in Sections \ref{sec:discussions} and \ref{sec:example}, a wide class of optimal control problems, e.g.\ for dynamical systems with logical constraints, can be formulated as CP$_\veps$, for a suitable design of the functions $J$ and $g$.
	For the sake of simplicity of the exposition, we assume the functions $g_h$ to have values in $\R$, which is without loss of generality\footnote{To model multiple constraints, $g_h$ can be designed as $g_h(x,\theta)=\max_{k\in\{1,...,n_\text{constr}\}} g_{h,k}(x,\theta)$, which again has values in $\R$.}. 
	
	As previously mentioned, solving CP$_{\veps}$ can be challenging even if $\P$ is known, since, in general, it would require to evaluate multidimensional integrals over the distribution describing $\theta$. In addition, note that the constraint function \eqref{eq:g} is generally nonlinear even if $g_h$ are linear, $h\in\{1,...,Z\}$, and the logical ``or'' can be formulate, e.g., by means of integer variables \cite{bemporad1999control}. For these two reasons, problems such as CP$_{\veps}$ can be, in general, intractable.

	\section{Uncertainty partitioning}\label{sec:partition}
	The underlying idea to solve CP$_\veps$ relies on partitioning the uncertainty domain $\Theta$. Our approach does not require assumptions on the distribution that induces the probability measure $\P$, nor its explicit knowledge. Instead, we require that a number $N\in\Z_{>0}$ of uncertainty realizations is available, with the intuition that a sufficiently large number of samples provides a good description of $\P$.

	\subsection{Partitioning scheme}
	Throughout the paper, we consider the following assumptions:
	\begin{assumption}\label{ass:iid}
		A dataset of $N$ independent and identically distributed realizations $\{\thetai\}_{i=1}^N$, drawn from $\P$, is available.
	\end{assumption}
	\begin{assumption}\label{ass:partition}
		There exists a set $\Dcal\subset\R^{n_\theta}$ such that: 1) $\Theta\subseteq\Dcal$; 2) $\Dcal$ can be partitioned into $K$ non-empty sets, i.e.: $\exists \Dcalj\subseteq\Dcal, \forall j\in\{1,...,K\}$, with $K\in\Z_{>0}$, such that: $\Dcalj\neq\emptyset, \forall j\in\{1,...,K\},$ $\Pr(\theta\in\Dcal^{(i)}\cap\Dcalj) = 0$, for all $i,j \in \{1,...,K\}, i\neq j$, and $\bigcup_{j=1}^{K} \Dcalj = \Dcal$. The sets $\Dcalj$ are assumed to be known, $\forall j\in\{1,...,K\}$.
	\end{assumption}
	
	We observe two aspects in Assumption \ref{ass:partition}. First, we introduce a set $\Dcal$ that contains the uncertainty domain $\Theta$ because, in some cases, we may not know the actual uncertainty domain exactly. However, depending on the application, we may be able to compute an over-approximation. For example, this can be the case if the uncertainty has a physical meaning, e.g., in temperature control \cite{shang2019datadriven}. Note that we have
	$\Pr(\theta\in\Theta)=1$ and $ \Pr(\theta\in\Dcal\setminus\Theta)=0$;
	hence, considering the superset $\Dcal$ instead of $\Theta$ will not affect the analysis (by requiring the functions $J, g_h, h\in\{1,...,Z\}$, to be defined on the superset $\Dcal$).
	In addition, while $\Theta$ can be any set, we can design $\Dcal$ with some convenient shape for computational reasons, as we will see later. Second, note that the condition $\Pr(\theta\in\Dcal^{(i)}\cap\Dcalj) = 0$ requires the sets $\Dcalj$ to be pairwise disjoint, up to 0-probability events. Thus, if $\P$ admits a density, which is the case of continuous distributions, any partition whose sets have pairwise disjoint interiors satisfies Assumption \ref{ass:partition}.
	In Section \ref{sec:discussions}, we will discuss three options to provide such a partition.
	
	We denote the sets of the indices corresponding to the realizations that belong to a given set by
	\begin{align*}
		\Ccalj := \{i\in\{1,...,N\}: \thetai\in\Dcalj \}, \ \forall j\in\{1,...,K\},
	\end{align*}
	where, in view of Assumption \ref{ass:partition}, they are pairwise disjoint with probability 1, and $\cup_{j=1}^K \Ccalj =\{1,...,N\}$. 
	For each set $\Dcalj$, we select a representative element $\hthetaj, \forall j\in\{1,...,K\}$. Although the choice is free, a possible one can be the average of the realizations that belong to $\Dcalj$, i.e.
	\begin{align}\label{eq:hthetaj}
		\hthetaj := \frac{1}{|\Ccalj|} \sum_{i\in\Ccalj} \thetai, \ \forall j\in\{1,...,K\},
	\end{align}
	for which we observe that $\hthetaj \in \Dcalj$ if $\Dcalj$ is convex.
	We define the probabilities $\hpj$ as the empirical probability mass associated to region $\Dcalj$, i.e.
	\begin{align}
		\begin{split}\label{eq:hpj}
			\hpj :&= \frac{1}{N} \sum_{i=1}^N  \mathbf{1}_{i\in\Ccalj}, \ \forall j\in\{1,...,K\},
		\end{split}
	\end{align}
	for which we have $\sum_{j=1}^{K} \hpj = 1$ with probability 1, in view of Assumption \ref{ass:partition}.
	Since the realizations $\{\thetai\}_{i=1}^N$ are drawn from $\P$ independently, the probability $\hpj$ is the approximation, obtained via counting, of the actual probability that $\theta$ belongs to $\Dcalj, \forall j\in\{1,...,K\}$, i.e.,
	\begin{align}\label{eq:prob_approx}
		\hpj \approx \Pr(\theta\in\Dcalj), \quad \forall j\in\{1,...,K\},
	\end{align}
	where the accuracy of the approximation increases with $N$.
	
	In the proposed approach we approximate the objective function in CP$_{\veps}$ by only employing the probabilities $\hpj$ and the representative elements $\hthetaj$. Thus, we optimize 
	\begin{align}\label{eq:cost_discrete}
		\sum_{j=1}^{K} \hpj J(x, \hthetaj),
	\end{align}
	which can be seen as the expectation of the cost function $J$, taken over the discrete distribution with domain on the set $\{\hthetaj\}_{j=1}^{K}$ and occurrence probabilities $\{\hpj\}_{j=1}^{K}$. 
	Concerning the chance constraint, we approximate \eqref{eq:cc} by
	\begin{align}\label{eq:cc_discr_approx}
		\sum_{j=1}^K \hpj \mathbf{1}_{[\lor_{h=1}^Z (g_h(x,\theta)\leq 0, \forall\theta\in\Dcalj)]} \geq 1-\veps,
	\end{align}
	where we optimally select for which of the sets in $\{\Dcalj\}_{j=1}^K$ the constraint is robustly satisfied at least by a function\footnote{If $g_h$ represents multidimensional constraints, i.e.\ $g_h:\Xcal\times\Dcal\rightarrow\R^{n_\mathrm{constr}}$, the robustification is performed equivalently element-wise, and the argument of the indicator function is: $\lor_{h=1}^Z \left(g_{h,i}(x,\theta)\leq 0, \forall\theta\in\Dcalj, \forall{i\in\{1,...,n_\mathrm{constr}\} }\right)$.} $g_h, h\in\{1,...,Z\}$,  in a way that the empirical probability mass contained in such sets is at least $1-\veps$.
	Then, by putting together \eqref{eq:cost_discrete}-\eqref{eq:cc_discr_approx}, and in view of the structure of $g$ in \eqref{eq:cc}, the Partition-based Problem reads as:
	\begin{align*}
		\text{PP}_\veps: \quad \min_{x\in\Xcal} \ & \sum_{j=1}^{K} \hpj J(x, \hthetaj)
		\\ \text{s.t.} \ & \sum_{j=1}^K \hpj \mathbf{1}_{[\lor_{h=1}^Z (g_h(x,\theta)\leq 0, \forall\theta\in\Dcalj)]} \geq 1-\veps, 
	\end{align*}
	with optimal value denoted by $J^\star_{\text{PP}_{\veps}}$ and feasible set
	\begin{align}
		\Fcal_{\text{PP}_\veps} := \{&x\in\Xcal:  \nonumber
		\\& \sum_{j=1}^K \hpj \mathbf{1}_{[\lor_{h=1}^Z (g_h(x,\theta)\leq 0, \forall\theta\in\Dcalj)]} \geq 1-\veps\} \label{eq:feas_set_pp}
	\end{align} 
	Given the previous considerations, we observe that PP$_{\veps}$ provides a discrete approximation, in terms of cost function and chance constraint, of the original CP$_{\veps}$, regardless of the nature of $\P$. Thus, in contrast to CP$_{\veps}$, PP$_{\veps}$ can be solved by a numerical solver, provided that the robust constraint in PP$_\veps$ can be efficiently computed (see Section \ref{sec:discussions} for more details). 
	In particular, the computational complexity of PP$_{\veps}$ is determined also by $K$, which is user-defined. Intuitively, a large $K$ results in a very accurate discrete approximation of the original problem, making the solution of PP$_{\veps}$ very close to the original CP$_{\veps}$, at the expenses of increased computational complexity. 
	
	Although its intuitive interpretation, there is no direct link between the optimal solution of PP$_{\veps}$ and the optimal solution of CP$_{\veps}$. This is due to the partitioning approximation, and due to the fact that $\{\hpj\}_{j=1}^K$ in \eqref{eq:prob_approx} are, in general, only an approximation of $\{\Pr(\theta\in\Dcalj)\}_{j=1}^K$. Thus, we provide next conditions to ensure that any solution resulting from PP$_{\veps}$ is feasible for the original CP$_{\veps}$, and we establish performance bounds between the optimal values $J^\star_{\text{CP}_{\veps}}$ and $J^\star_{\text{PP}_{\veps}}$.

	\subsection{Probabilistic error}
	As previously stated, $\P$ does not need to be known. Thus, the feasibility analysis is purely data-driven, and it only relies on the accuracy of the estimated probabilities $\hpj, \forall j\in\{1,...,K\}$, which become more and more accurate as more samples are available. Let the following assumption hold:
	\begin{assumption}\label{ass:uniform:M}
		Let $\delta\in(0,\veps]$, $\beta\in(0,1]$. We assume that $N, K, \delta, \beta$ satisfy:
		\begin{align*}
			N \geq \frac{K \log 2 + \log\frac{1}{\beta}}{2\delta^2}.
		\end{align*}
	\end{assumption}
	
	Intuitively, the parameter $\delta$ represents the level of accuracy that we would like to achieve in the estimates for the probabilities $\Pr(\theta\in\Dcalj), \forall j\in\{1,...,K\}$, which, as for $K$, is user-defined, provided that the given condition for $N$ is satisfied. The parameter $\beta$, instead, is a confidence parameter, which encodes the fact that the probabilities $\{\hpj\}_{j=1}^K$ are estimated empirically, thus their computation is affected by the variability of the available data. Note that smaller values for $\delta$ and $\beta$ require a larger $N$.
	
	The condition on $N$ in Assumption \ref{ass:uniform:M} resembles other common assumptions on the sample complexity, especially in the field of sampling-based methods \cite{luedtke2008sample}, or randomized algorithms for non-convex programs \cite{alamo2009randomized, tempo2013randomized}. However, an important difference is that our requirement for $N$ is linear in $K$, which is a user-defined parameter. On the other side, in \cite{alamo2009randomized}, \cite{luedtke2008sample, campi2011samplinganddiscarding}, the sample complexity is affected by the number of decision variables in the optimization problem, which can be high especially in non-convex settings \cite{mohajerinesfahani2015performance}.

	Under Assumption \ref{ass:uniform:M}, we can state the following lemma, which indeed provides a concentration bound for the approximation \eqref{eq:prob_approx}.
	\begin{lemma}\label{lemma:uniform}
		Let $\delta\in(0,\veps)$, $\beta\in(0,1]$. Under Assumptions \ref{ass:iid}-\ref{ass:uniform:M}, it holds that
		\begin{align*}
			\P^N\left( \max_{\Jcal\subseteq\{1,...,K\}}
			\left| \sum_{j\in\Jcal} \left(\Pr(\theta\in\Dcalj) - \hpj  \right) \right|
			\leq\delta\right) \geq 1-\beta .
		\end{align*}
	\end{lemma}
	\begin{proof}
		See Appendix A.
	\end{proof}
	
	We observe that Lemma \ref{lemma:uniform} provides a uniform bound for the estimates $\hpj, \forall j\in\{1,...,K\}$. Indeed, let $\Bar{\Dcal}$ be a set obtained by taking the union of some of the sets in $\{\Dcalj\}_{j=1}^{K}$, that is: $\Bar{\Dcal} = \cup_{j\in\Jcal}\Dcalj$, for some $\Jcal\subseteq\{1,...,K\}, \Jcal\neq\emptyset$. Then, Lemma \ref{lemma:uniform} ensures that the distance between the empirical and the true probability that $\theta$ belongs to $\Bar{\Dcal}$ is at most $\delta$. This bound is indeed uniform, since it holds with confidence at least $1-\beta$ for any possible choice for $\Bar{\Dcal}$ , i.e., for all the possible combinations to construct a set $\Jcal$ as a subset of $\{1,...,K\}$.

	\section{Feasibility and performance analysis}\label{sec:main_section}
	Building upon the previous lemma, we derive results about the feasibility of PP$_\veps$, as well as its optimality. More specifically, we seek sufficient conditions such that an appropriate constraint modification in PP$_\veps$ ensures that $\Fcal_{\text{PP}_{\veps}} \ \subseteq \ \Fcal_{\text{CP}_\veps}$. In addition, we are interested in finding an interval (easy to compute and not excessively conservative) that contains $J^\star_{\text{CP}_\veps}$ with high confidence.

	\subsection{Probabilistic feasibility}\label{subsec:feas}
	
	Given the result in Lemma \ref{lemma:uniform}, we can provide a feasibility result for PP$_\veps$:
	\begin{theorem}\label{th:feas_prob}
		Let $\delta\in(0,\veps]$, $\beta\in(0,1]$. Under Assumptions \ref{ass:iid}-\ref{ass:uniform:M}, it holds that
		\begin{align*}
			\P^N \left(\Fcal_{\text{PP}_{\veps-\delta}} \ \subseteq \ \Fcal_{\text{CP}_\veps}\right) \geq 1-\beta,
		\end{align*}
		i.e., any feasible solution for PP$_{\veps-\delta}$ is feasible for CP$_\veps$ with probability at least $1-\beta$.
	\end{theorem}
	\begin{proof}
		In view of Assumption \ref{ass:partition}, and leveraging the law of total probability, we can write, for all $x\in\Xcal$:
		\begin{align}
			&\Pr(g(x,\theta)\leq0)  \nonumber
			\\& = \sum_{j=1}^{K} \Pr(g(x,\theta)\leq0| \theta\in\Dcalj)\Pr(\theta\in\Dcalj)  \nonumber
			\\& \geq \sum_{j=1}^{K} \mathbf{1}_{[\lor_{h=1}^Z (g_h(x,\theta)\leq 0, \forall\theta\in\Dcalj)]} \Pr(\theta\in\Dcalj)  \label{proof:feas_prob:1}
		\end{align}
		which holds since, when indicator function takes value 1 (for some $j\in\{1,...,K\}$), there exists $\Bar{h}\in \{1,...,Z\}$ such that 
		$$g_{\Bar{h}}(x,\theta)\leq 0, \forall\theta\in\Dcalj,$$
		which implies that 
		$$ g(x,\theta) =  \min_{h\in\{1,...,Z\}} g_h(x,\theta) \leq g_{\Bar{h}}(x,\theta)\leq 0, \forall\theta\in\Dcalj. $$
		Thus, whenever, in \eqref{proof:feas_prob:1}, the indicator function takes value 1 for some $j\in\{1,...,K\}$, we have
		$\Pr(g(x,\theta)\leq0| \theta\in\Dcalj)=1$ (the case in which the indicator takes value 0 is trivial).
		
		Now, we will prove the claim leveraging Lemma \ref{lemma:uniform}. Under Assumptions \ref{ass:iid} and \ref{ass:uniform:M}, Lemma \ref{lemma:uniform} guarantees that the following hold uniformly for all $x\in\Xcal$, with probability at least $1-\beta$:
		\begin{align}
				\delta & \geq \max_{\Jcal\subseteq\{1,...,K\}}
				\left| \sum_{j\in\Jcal} \left(\hpj - \Pr(\theta\in\Dcalj)  \right) \right|  \nonumber
				\\& \geq \left| \sum_{j=1}^{K} \mathbf{1}_{[\lor_{h=1}^Z (g_h(x,\theta)\leq 0, \forall\theta\in\Dcalj)]} (\hpj - \Pr(\theta\in\Dcalj)) \right| \nonumber
				\\& \geq \sum_{j=1}^{K} \mathbf{1}_{[\lor_{h=1}^Z (g_h(x,\theta)\leq 0, \forall\theta\in\Dcalj)]} ( \hpj-\Pr(\theta\in\Dcalj) ) \label{proof:feas_prob:4}
		\end{align}
		where the first inequality is from Lemma \ref{lemma:uniform}, whereas the second and the third one are always true.
		Then, any $x\in\Fcal_{\text{PP}_{\veps-\delta}}$ satisfies, with probability at least $1-\beta$
		\begin{align*}
			\delta & \geq \sum_{j=1}^{K} \mathbf{1}_{[\lor_{h=1}^Z (g_h(x,\theta)\leq 0, \forall\theta\in\Dcalj)]} \left(\hpj-\Pr(\theta\in\Dcalj)\right) 
			\\& \geq 1-(\veps-\delta) - \sum_{j=1}^{K} \mathbf{1}_{[\lor_{h=1}^Z (g_h(x,\theta)\leq 0, \forall\theta\in\Dcalj)]}\Pr(\theta\in\Dcalj)
			\\& \geq 1-(\veps-\delta) - \Pr(g(x,\theta)\leq0),
		\end{align*}
		where the first inequality is from from \eqref{proof:feas_prob:4}, the second holds since we consider $x\in\Fcal_{\text{PP}_{\veps-\delta}}$, and the third comes from \eqref{proof:feas_prob:1}.
		Hence, in view of Lemma \ref{lemma:uniform}, the previous implications translate to
		\begin{align*}
			& \P^N \left( \Pr(g(x,\theta)\leq0) \geq 1-\veps, \ \forall x\in\Fcal_{\text{PP}_{\veps-\delta}} \right) \\& \geq \P^N\left( \max_{\Jcal\subseteq\{1,...,K\}}
			\left| \sum_{j\in\Jcal} \left(\Pr(\theta\in\Dcalj) - \hpj  \right) \right|
			\leq\delta\right)
			\\& \geq 1-\beta,
		\end{align*}
		which proves the claim.
	\end{proof}
	
	In view of Theorem \ref{th:feas_prob}, any choice of $N, K, \delta$ ensures that PP$_{\veps-\delta}$ is also feasible for CP$_{\veps}$ with high probability, provided that $N$ satisfies Assumptions \ref{ass:partition} and \ref{ass:uniform:M}. Intuitively, with a reduction of the risk parameter $\veps$ of an amount $\delta$, we can account for the probabilistic error introduced by the counting procedure.
	In addition, we remark that despite the bound in Lemma \ref{lemma:uniform} is \emph{a posteriori}, i.e., for a given partition, we evaluate the accuracy of the estimate for $\{\hpj\}_{j=1}^K$, the feasibility result in Theorem \ref{th:feas_prob} is \emph{a priori}, i.e., no validation set is required to evaluate the risk of constraint of the solution of PP$_{\veps-\delta}$.

	\subsection{Probabilistic performance bounds}\label{subsec:performance}
	In several applications, e.g.\ when the objective function represents a cost or a consumption, optimal performances are also of interest in addition to feasibility guarantees. 
	Thus, the goal of this section is to provide an interval $[\underline{c}, \Bar{c}]$ such that:
	\begin{align}\label{eq:bound_Jccp}
		J^\star_{\text{CP}_\veps} \in [\underline{c}, \Bar{c}],
	\end{align}
	with $\underline{c}, \Bar{c}, \in\R_{\geq0}$ efficiently computable and at the same time not too conservative. 

	Due to the robustification of both the constraint and the risk parameter, PP$_{\veps-\delta}$ provides feasible solutions for CP$_\varepsilon$ in view of Theorem \ref{th:feas_prob}. For this reason, it is natural to think that $\Bar{c}$ will be related to $J^\star_{\text{PP}_{\veps-\delta}}$. Following this intuition, we introduce an auxiliary optimization problem, where, in contrast to PP$_\veps$, the constraint and the risk parameter are relaxed. By recalling the definition of $g$ in \eqref{eq:cc}, we introduce the following Relaxed Problem (RP):
	\begin{align*}
		\text{RP}_\veps: \ \min_{x\in\Xcal} \ & \sum_{j=1}^{K} \hpj J(x, \hthetaj)
		\\ \text{s.t.} \ & \sum_{j=1}^K \hpj \mathbf{1}_{[\lor_{h=1}^Z (g_h(x, \hthetaj) \leq \gammaj_h(x,\hthetaj))]} \geq 1-\veps 
	\end{align*}
	with optimal value $J^\star_{\text{RP}_{\veps}}$ and feasible set
	\begin{align}
		\Fcal_{\text{RP}_\veps} = \{&x\in\Xcal: \label{eq:feas_set_rp}
		\\& \sum_{j=1}^K \hpj \mathbf{1}_{[\lor_{h=1}^Z (g_h(x, \hthetaj) \leq \gammaj_h(x,\hthetaj))]} \geq 1-\veps\}, \nonumber
	\end{align}
	where the functions $\gammaj_h:\Xcal\times\Dcal \rightarrow \R_{\geq0}$ satisfy, $\forall j\in\{1,...,K\}, \forall h\in\{1,...,Z\}$:
	\begin{align}\label{eq:gamma}
		\begin{split}
			&g_h(x,\hthetaj) - g_h(x, \theta) \leq \gammaj_h(x,\hthetaj),
			\\& \forall x\in\Xcal, \forall\theta\in\Dcalj.
		\end{split}
	\end{align}
	In contrast to PP$_\veps$, where the constraint function is tightened, in RP$_\veps$ each function $g_h$ is relaxed by an appropriate amount.
	Note that, given the above condition, $\gammaj_h$ is a non-negative function. Also, note that, according to \eqref{eq:gamma}, $\gammaj_h$ provides a bound on the growth of the functions $g_h$, with $h\in\{1,...,Z\}$, and, for example, it is satisfied for functions $g_h$ that are locally Lipschitz continuous in the second argument, since this would imply, for some function $L_h:\Xcal\rightarrow\R_{\geq 0}$
	\begin{align*}
		& g_h(x,\hthetaj) - g_h(x, \theta) 
		\\& \leq L_h(x) \max_{\theta\in\Dcalj} \|\hthetaj-\theta\| =: \gammaj_h(x,\hthetaj).
	\end{align*}

	From Theorem \ref{th:feas_prob}, we know that $\Fcal_{\text{PP}_{\veps-\delta}}\subseteq\Fcal_{\text{CP}_{\veps}}$ with high confidence. Interestingly, an analogous inclusion relation holds for the relaxed problem as well:
	\begin{lemma}\label{lemma:inclusion}
		Let $\delta\in(0,\veps]$, $\beta\in(0,1]$. Under Assumptions \ref{ass:iid}-\ref{ass:uniform:M}, it holds that
		\begin{align}\label{lemma:inclusion:claim}
			\P^N \left(\Fcal_{\text{PP}_{\veps-\delta}} \ \subseteq \ \Fcal_{\text{CP}_\veps} \ \subseteq \ \Fcal_{\text{RP}_{\veps+\delta}} \right) \geq 1-\beta.
		\end{align}
	\end{lemma}
	\begin{proof}
		See Appendix A.
	\end{proof}
	Note that the first inclusion has been proved in Theorem \ref{th:feas_prob}. Here it is repeated because we highlight that the inclusions in \eqref{lemma:inclusion:claim} hold \emph{jointly} with probability at least $1-\beta$. The inclusion relation in \eqref{lemma:inclusion:claim} provides useful insights on how to provide a performance bound such as \eqref{eq:bound_Jccp}, since an inclusion relation between feasible sets has a direct effect on the optimality of the obtained solution. 

	We introduce the following assumption on the cost function:
	\begin{assumption}\label{ass:Jlip_1}
		We assume that the function $J$ is Lipschitz continuous, in particular:
		$$|J(x, \theta_1) - J(x, \theta_2)| \leq L_\theta \|\theta_1 - \theta_2\|_q , \forall x\in\Xcal,\ \forall \theta_1, \theta_2\in\Dcal,$$
		$$|J(x_1, \theta) - J(x_2, \theta)| \leq L_x \|x_1 - x_2\|_2 , \forall \theta\in\Dcal,\forall x_1, x_2\in\Xcal,$$
		for some constants $L_\theta, L_x \in\R_{>0}$.
	\end{assumption}
	As we will see in Section \ref{subsec:pwa_details}, this assumption is satisfied in several cases, e.g., in the case of optimal control of nonlinear systems with Lipschitz dynamics and Lipschitz stage cost (e.g.\ the 1-norm, which is globally Lipschitz, or also the squared 2-norm, which is locally Lipschitz over bounded sets), and it is common when providing performance bounds, see e.g.\ \cite{kleywegt2002sample}, or \cite{rockafellar2009variational} Example 7.62. Also, we specify that $L_\theta$ is the Lipschitz constant with respect to a given $q$-norm, which can be chosen conveniently, e.g.\ based on the nature of the cost function $J$. Conversely, $L_x \in\R_{>0}$ is the Lipschitz constant specifically with respect to the $2$-norm, which is needed in the following derivations. This is without loss of generality, since all norms are equivalent over a finite-dimensional space.
	
	Let us state a preliminary lemma:
	\begin{lemma}\label{lemma:uniform_concentration}
		Let $\beta\in(0,1]$, and let Assumptions \ref{ass:iid} and \ref{ass:Jlip_1} hold. Let $D = \max_{\theta_1, \theta_2\in\Dcal} \|\theta_1 - \theta_2\|_q, R=\max_{x_1,x_2\in\Xcal}\|x_1-x_2\|_2$, and choose  $r\in(0, R]$. 
		Define $c_1$ and $c_2$ as
		\begin{align}\label{eq:c1,c2}
			\begin{split}
				&c_1 = \sqrt{\frac{L_\theta^2 D^2}{2N}\left(\log\frac{1}{\beta} + n\log\frac{3R}{r} \right)}\ , \ \ 
				c_2 = 2L_x r.
			\end{split}
		\end{align}
		Then, with probability at least $1-2\beta$, it holds that
		$$\max_{x\in\Xcal} \left|\E [J(x,\theta)] - \frac{1}{N} \sum_{i=1}^N J(x,\thetai)\right| \leq c_1 + c_2.$$
	\end{lemma}
	\begin{proof}
		See Appendix A.
	\end{proof}
	
	We can now provide our main result about optimal values:
	\begin{theorem}\label{th:optimal_values}
		Let $\delta\in(0,\veps]$, $\beta\in(0,1]$, and let Assumptions \ref{ass:iid}-\ref{ass:Jlip_1} hold. Let $D = \max_{\theta_1, \theta_2\in\Dcal} \|\theta_1 - \theta_2\|_q, R=\max_{x_1,x_2\in\Xcal}\|x_1-x_2\|_2$, and choose  $r\in(0, R]$. Define $c_1$ and $c_2$ as in Lemma \ref{lemma:uniform_concentration}, and $c_3$ as
		\begin{align*}
			c_3 = \frac{1}{N}\sum_{j=1}^{K} \sum_{i\in\Ccalj} L_{\theta} \|\hthetaj - \thetai\|_q.
		\end{align*}
		Then, with $c=c_1+c_2+c_3$, 
		\begin{align}\label{th:optimal_values:claim}
			\P^N\left(J^\star_{\text{RP}_{\veps+\delta}} - c \leq J^\star_{\text{CP}_{\veps}} \leq J^\star_{\text{PP}_{\veps-\delta}} + c \right) \geq 1-3\beta.
		\end{align}
	\end{theorem}
	\begin{proof}
		For a given $x\in\Xcal$, we have, in view of the triangle inequality
		\begin{align*}
			&\Big|\E [J(x,\thetai)]  - \sum_{j=1}^K \hpj J(x,\hthetaj)\Big|
			\\& \leq \begin{aligned}[t]
				&\Big|\E [J(x,\theta)]  - \frac{1}{N}\sum_{i=1}^N J(x,\thetai)\Big|
				\\& + \Big|\frac{1}{N}\sum_{i=1}^N J(x,\thetai)  - \sum_{j=1}^K \hpj J(x,\hthetaj)\Big|.
			\end{aligned} 
		\end{align*}
		To bound the first term we can apply Lemma \ref{lemma:uniform_concentration}, whereas for the second we have, in view of the definition for $\hpj$ in \eqref{eq:hpj} and using the triangle inequality
		\begin{align*}
			&\Big|\frac{1}{N}\sum_{i=1}^N J(x,\thetai)  - \sum_{j=1}^K \hpj J(x,\hthetaj)\Big|
			\\&= \Big|\frac{1}{N}\sum_{j=1}^K \sum_{i\in\Ccalj} J(x,\thetai)  - \sum_{j=1}^K\sum_{i\in\Ccalj} \frac{1}{N} J(x,\hthetaj)\Big|
			\\& \leq \frac{1}{N} \sum_{j=1}^K\sum_{i\in\Ccalj}|J(x,\thetai) - J(x,\hthetaj)|
			\\& \leq \frac{1}{N}\sum_{j=1}^{K} \sum_{i\in\Ccalj} L_{\theta} \|\thetai - \hthetaj\|_q = c_3,
		\end{align*}
		where the last relation holds in view Assumption \ref{ass:Jlip_1}. Hence, together with Lemma \ref{lemma:uniform_concentration}, the following uniform concentration bound
		\begin{align}\label{proof:optimal_values:bound}
			\max_{x\in\Xcal}\Big|\E [J(x,\theta)] - \sum_{j=1}^K \hpj J(x,\hthetaj)\Big| \leq c
		\end{align}
		holds with probability at least $1-2\beta$, where $c=c_1+c_2+c_3$. Together with Lemma \ref{lemma:inclusion}, a union bound yields:
		\begin{align}\label{proof:optimal_values:inclusion}
			\begin{split}
				\P^N\Big(&\max_{x\in\Xcal}\Big|\E [J(x,\theta)] - \sum_{j=1}^K \hpj J(x,\hthetaj)\Big| \leq c , 
				\\& \Fcal_{\text{PP}_{\veps-\delta}} \subseteq \Fcal_{\text{CP}_\veps} \subseteq \Fcal_{\text{RP}_{\veps+\delta}} \Big) \geq 1-3\beta.  
			\end{split}
		\end{align}
		Then, the uniform bound and the inclusions in \eqref{proof:optimal_values:inclusion} determine the following ordering relation between the optimal values:
		$$J^\star_{\text{RP}_{\veps+\delta}} - c \leq J^\star_{\text{CP}_{\veps}} \leq J^\star_{\text{PP}_{\veps-\delta}} + c,$$
		which holds at least with probability $1-3\beta$.
	\end{proof}
	
	Note that Theorem \ref{th:optimal_values} bounds $J^\star_{\text{CP}_{\veps}}$, which in general is difficult to compute even for simple functions $J$ and $g$, with two quantities that are easy to compute whenever the problems PP$_\veps$ and RP$_{\veps}$ can be solved efficiently, which is the case for a large class of functions $J$ and $g$ (see Section \ref{sec:discussions}).
	
	The bound in \eqref{th:optimal_values:claim} depends on three components: $c_1$ is the error due to sampling and it is decreasing in $N$; $c_2$ describes the variability of the function $J$, and it can be made small by choosing $r$ small at the cost of a mild increase in the term $c_1$ (note the logarithmic dependence); and $c_3$ is a measure of the accuracy of the partition, which becomes smaller and smaller as $K$ increases. We also note that $c_1$ depends also on the number of decision variables $n$. This seems to be a common feature when providing uniform probabilistic bounds as in Lemma \ref{lemma:uniform_concentration}, since the solution of PP$_{\veps-\delta}$ and RP$_{\veps+\delta}$ is in principle data-dependent.
	Note that, if only feasibility is of interest, or if nominal performances are optimized in CP$_\veps$, $n$ would not play any role in the proposed results.

	\subsection{Local continuity of the optimal value}
	
	Theorem \ref{th:optimal_values} provides the sought bound \eqref{eq:bound_Jccp}, which can be obtained by solving $\text{PP}_{\veps-\delta}$ and $\text{RP}_{\veps+\delta}$. In this section, we derive a theoretical bound for 
	\begin{align}\label{eq:detBound}
		|J^\star_{\text{PP}_{\veps-\delta}} -  \ J^\star_{\text{CP}_\veps}|,
	\end{align}
	in order to find a mathematical expression that characterizes the source of the optimality gap between the problem that we actually solve, i.e.\ PP$_{\varepsilon-\delta}$, and the problem that we would ideally solve, i.e.\ CP$_{\varepsilon}$, without going through the solution of the auxiliary problem RP$_{\varepsilon+\delta}$.

	Bounding \eqref{eq:detBound} is in principle difficult, since a direct study of $|J_{\text{CP}_\veps} - J_{\text{PP}_{\veps-\delta}}|$ requires to evaluate the distance between $\Fcal_{\text{CP}_{\veps}}$ and $\Fcal_{\text{PP}_{\veps-\delta}}$, which can be difficult since $\Fcal_{\text{CP}_{\veps}}$ can be a very complex set (possibly non-convex) for a general distribution $\P$. However, we observe that deriving an optimality bound for 
	\begin{align}\label{eq:bound_rp_pp}
		J^\star_{\text{PP}_{\veps-\delta}} - J^\star_{\text{RP}_{\veps+\delta}},
	\end{align}
	will lead to a bound for \eqref{eq:detBound}, by combining \eqref{th:optimal_values:claim} with the bound for \eqref{eq:bound_rp_pp} (note that \eqref{eq:bound_rp_pp} is always non-negative). In particular, since problems PP$_{\veps-\delta}$ and RP$_{\veps+\delta}$ optimize the same cost function, the performance bound \eqref{eq:bound_rp_pp} depends only on the distance between the constraint sets, i.e.\ respectively $\Fcal_{\text{PP}_{\veps-\delta}}$ and $\Fcal_{\text{RP}_{\veps+\delta}}$, which have a similar structure in view of the discrete approximation in PP$_{\veps-\delta}$ and RP$_{\veps+\delta}$.
	
	Leveraging tools from variational analysis \cite{rockafellar2009variational}, in order to evaluate how much the constraints set differ from each other, we consider the following definition:
	\begin{definition}\label{def:hausd}(Hausdorff distance, cf. \cite{rockafellar2009variational}, Chapter 4).
		Let $\Acal, \Bcal \subset \R^n$ be closed and non-empty.
		The Hausdorff distance between $\Acal$ and $\Bcal$ is defined as
		\begin{align*}
			\mathbb{d}(\Acal, \Bcal) := \inf \left\{\rho \geq 0 : \Acal\subseteq\Bcal \oplus \rho \B , \Bcal\subseteq\Acal \oplus \rho \B \right\}.
		\end{align*}
	\end{definition}
	As we will see in the next lemma, \eqref{eq:bound_rp_pp} can be bounded in terms of the Hausdorff distance between $\Fcal_{\text{RP}_{\veps+\delta}}$ and $ \Fcal_{\text{PP}_{\veps-\delta}}$.
	\begin{lemma}\label{lemma:bound_hdist}
		Let us assume that $\Fcal_{\text{RP}_{\veps+\delta}}, \Fcal_{\text{PP}_{\veps-\delta}}$ are non-empty.
		Then, it holds that
		\begin{align}\label{lemma:optimality_hdist_claim0}
			0 \leq J^\star_{\text{PP}_{\veps-\delta}} - J^\star_{\text{RP}_{\veps+\delta}} 
			\leq L_x \mathbb{d}(\Fcal_{\text{PP}_{\veps-\delta}}, \Fcal_{\text{RP}_{\veps+\delta}}).
		\end{align}
	\end{lemma}
	\begin{proof}
		See Appendix A.
	\end{proof}
	Note that this bound is deterministic, i.e., it holds for any possible sample set $\{\thetai\}_{i=1}^N$, since once a sample set $\{\thetai\}_{i=1}^N$ is given, PP$_{\veps-\delta}$ and RP$_{\veps+\delta}$ are deterministic optimization problems.
	
	The Hausdorff distance between sets is not always easy to compute or bound. Thus, in this section, we focus on a specific type of constraint function defined by
	\begin{align}\label{eq:g_linear}
		g_h(x,\theta) = C_h x + D_h\theta + b_h, \ \forall h\in\{1,...,Z\},
	\end{align}
	where $C_h, D_h$, and $b_h$, are matrices and vectors of appropriate dimension, $\forall h\in\{1,...,Z\}$. As we will see in Section \ref{sec:example}, this class of constraints includes, e.g., controlled dynamical systems with logical constraints (e.g., PWA systems), with additive uncertainty.
	
	In this setting, we observe that the robustified chance constraint in PP$_\veps$ reads as:
	$$\sum_{j=1}^K \hpj \mathbf{1}_{[\lor_{h=1}^Z (C_h x + D_h\theta + b_h \leq 0, \forall \theta\in\Dcalj)]} \geq 1-\veps ,$$
	and for constraint functions as in \eqref{eq:g_linear} we have
	\begin{align}\label{eq:robust_linear}
		\Fcal_{\text{PP}_{\veps}} = 
		\{&x\in\Xcal:	
		\\&\sum_{j=1}^K \hpj \mathbf{1}_{[\lor_{h=1}^Z (C_h x + D_h\hthetaj + b_h \leq -\tau^{(j)}_{h})]} \geq 1-\veps\}, \nonumber
	\end{align}
	where, for $j\in\{1,...,K\}$ and $h\in\{1,...,Z\}$, the $l$-th component of the vector $\tau^{(j)}_{h}$ satisfies
	\begin{align}\label{eq:tight_linear}
		\tau^{(j)}_{h, l} = \max_{\theta\in\Dcalj} D_{h, l} (\theta-\hthetaj),
	\end{align}
	where $D_{h, l}$ is the $l$-th row of the matrix $D_h$. Similarly, from \eqref{eq:gamma}, we have
	\begin{align}
		\Fcal_{\text{RP}_{\veps}} = 
		\{&x\in\Xcal:	\label{eq:relaxed_linear}
		\\&\sum_{j=1}^K \hpj \mathbf{1}_{[\lor_{h=1}^Z (C_h x + D_h\hthetaj + b_h \leq \gammaj_{h})]} \geq 1-\veps\}, \nonumber
	\end{align}
	with
	\begin{align}\label{eq:gamma_linear}
		\gammaj_{h, l} = \max_{\theta\in\Dcalj} D_{h, l} (\hthetaj - \theta),
	\end{align}
	thus by relaxing each row of $g_h, h\in\{1,...,Z\}$.
	Essentially, for the special constraint function in \eqref{eq:g_linear}, the robust and relaxed constraint in PP$_\veps$, RP$_\veps$, are reformulated in terms of a representative element for each region $\Dcalj$, which in our case is $\thetaj, j\in\{1,...,K\}$, and a correction term in the right-hand side of the inequality.
	In particular, if $\Dcalj$ are polytopic sets, the maximization in \eqref{eq:tight_linear} and \eqref{eq:gamma_linear} can be found explicitly via linear programming. In view of \eqref{eq:robust_linear} and \eqref{eq:relaxed_linear}, the feasible set $\Fcal_{\text{PP}_\veps}$ of PP$_\veps$, and the feasible set $\Fcal_{\text{RP}_\veps}$ of RP$_\veps$, are the union of, at most, $M=2^{K+Z}$ polytopic sets, i.e.:
	\begin{align}\label{proof:hdist_bound:tFcal}
		\Fcal_{\text{PP}_\veps} = \bigcup_{k=1}^M \Fcal_{\text{PP}_\veps}^k, \quad
		\Fcal_{\text{RP}_\veps} = \bigcup_{k=1}^M \Fcal_{\text{RP}_\veps}^k, 
	\end{align}
	with 
	\begin{align}
		\Fcal_{\text{PP}_\veps}^k = \{x: \ & C_h x + D_h\hthetaj + b_h \leq -\tau_{h}^{(j)},	\label{proof:hdist_bound:tF_PP_k}
		\\& \forall h\in\mathcal{H}_k, \forall j\in\mathcal{J}_k \}, \nonumber
	\end{align}
	\begin{align}
		\Fcal_{\text{RP}_\veps}^k = \{x: \ & C_h x + D_h\hthetaj + b_h \leq \gammaj_h, \label{proof:hdist_bound:tF_RP_k}
		\\& \forall h\in\mathcal{H}_k, \forall j\in\mathcal{J}_k \}\}, \nonumber
	\end{align}
	for all possible combinations of $\mathcal{H}_k\subseteq\{1,...,Z\}$ and $ \mathcal{J}_k \subseteq\{1,...,K\}$ such that 
	$\sum_{j\in\Jcal_k} \hpj \geq 1-\veps$ (thus in number at most $M$). Note that, even if the single sets in the unions in \eqref{proof:hdist_bound:tFcal} are polytopes, the resulting unions are in general non-convex.
	
	First, we provide an analytical bound on the Hausdorff distance for union of polytopes as in \eqref{proof:hdist_bound:tFcal}.
	\begin{lemma}\label{lemma:Hdist_polytopes}
		Let us assume that $\Fcal_{\text{PP}_\veps}^k, \Fcal_{\text{RP}_\veps}^k$ are non-empty, $\forall k\in\{1,...,M\}$, and consider $\Fcal_{\text{PP}_\veps}, \Fcal_{\text{RP}_\veps}$ defined as in \eqref{proof:hdist_bound:tFcal}-\eqref{proof:hdist_bound:tF_RP_k}. Then, we have:
		$$\d(\Fcal_{\text{PP}_\veps}, \Fcal_{\text{RP}_\veps}) \leq \max_{h\in\{1,...,Z\}}\max_{j\in\{1,...,K\}} \frac{\|\gamma_{h}^{(j)} - \tau_{h}^{(j)}\|_2}{\sigma(C_h)},$$
		where $\sigma(C_h)$ is the minimum of the smallest singular values, where the minimum is taken over all the square invertible submatrices of $C_h$.
	\end{lemma}
	\begin{proof}
		See Appendix A.
	\end{proof}
	Note that Lemma \ref{lemma:Hdist_polytopes} states that the Hausdorff distance between $\Fcal_{\text{PP}_\veps}$ and $\Fcal_{\text{RP}_\veps}$ is bounded by the worst-case distance between each component of the union in \eqref{proof:hdist_bound:tFcal}.
	
	We can now provide the following result:
	\begin{lemma}\label{lemma:optimality_hdist}
		 Let Assumptions \ref{ass:iid}-\ref{ass:Jlip_1} hold. Let us assume that the functions $g_h$ are of the type \eqref{eq:g_linear}, for $\forall h\in\{1,...,Z\}$, and that  $\Fcal_{\text{PP}_\veps}^k, \Fcal_{\text{RP}_\veps}^k$ are non-empty, $\forall k\in\{1,...,M\}$, and suppose that 
		\begin{align}\label{lemma:feas_set_jump:assum}
			\veps \neq \sum_{j\in\Jcal} \hpj, \ \forall \Jcal\subset\{1,...,K\}.
		\end{align}
		Then, for any $\delta$ in the interval
		\begin{align}\label{lemma:feas_set_jump:delta_interval}
			\Biggl(0, \ \min_{\Jcal\subset\{1,...,K\}} \biggl|\veps - \sum_{j\in\Jcal} \hpj \biggr|\Biggr),
		\end{align}
		with probability at least $1-3\beta$, it holds that
		\begin{align}\label{lemma:optimality_hdist_claim}
			&J^\star_{\text{CP}_{\veps}} - J^\star_{\text{PP}_{\veps-\delta}}  \in
			\left[ - L_x \max_{\substack{j\in\{1,...,K\} \\ h\in\{1,...,Z\}}} \frac{\|\tau^{(j)}_{h}-\gammaj_h\|_2}{\sigma(C_h)} - c, \ c \right]. 
		\end{align}
	\end{lemma}
	\begin{proof}
		See Appendix A.
	\end{proof}

	\begin{figure}
		\centering
		\includegraphics[scale=1]{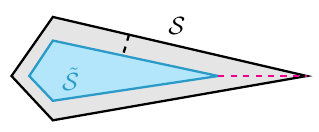}
		\caption{A 2-dimensional polytope $\Scal:=\{x:Ax\leq b\}$, where $\sigma(A)$ is small. A small tightening (black dashed line) originates $\tScal$, and it may still lead to high distance between $\Scal$ and $\tScal$ (pink dashed line), and thus a potentially large optimality gap.}
		\label{fig:hdist}
	\end{figure}
	
	We observe the following aspects deriving from Lemma \ref{lemma:optimality_hdist}:
	\begin{enumerate}
		\item In general, the value of a chance-constrained program with a discrete distribution can be discontinuous with respect to changes in $\varepsilon$. Condition \eqref{lemma:feas_set_jump:delta_interval} ensures that RP$_{\varepsilon - \delta}$ and RP$_{\varepsilon + \delta}$ have the same solution for a small enough $\delta$,
		which can be interpreted as a local continuity of the optimal value of a chance-constrained program under small perturbation of the risk parameter.
		\item Interestingly, \eqref{lemma:optimality_hdist_claim} highlights that the performance bound not only depends on the partitioning, but also on the geometry of the constraint set described by the functions $g_h, h\in\{1,...,Z\}$. Indeed, even for small tightening or relaxation, the optimality gap can still be large if the quantity $\sigma(C_h)$ is very small, which can be the case for sets with an ``imbalanced'' shape (i.e., very narrow in a certain direction). See Figure \ref{fig:hdist} for an example in which a small tightening can lead to a potentially large optimality gap. Computing $\sigma(C_h)$ is a handy tool to evaluate the geometry of multidimensional sets, guiding the choice of a larger $K$ if needed.
		\item In principle, \eqref{lemma:optimality_hdist_claim} allows to derive an interval that contains $J^\star_{\text{CP}_\veps}$ without solving the auxiliary problem RP$_{\varepsilon+\delta}$. Although this allows to save computational budget, it may be significantly more conservative, since the max in \eqref{lemma:optimality_hdist_claim} essentially tells that the optimality gap is obtained as the worst-case distance between each $\Fcal_{\text{PP}_\veps}^k$ and the corresponding $\Fcal_{\text{RP}_\veps}^k$.
	\end{enumerate}

	\section{Discussion on the results and practical implementation}\label{sec:discussions}
	In this section, we discuss the tractability of the proposed approach, with a particular emphasis on the robust constraint in PP$_\veps$. In addition, we provide details on how to implement it on a case study of interest, i.e.\ MPC for PWA systems, with a focus on deriving a partitioning for such application. 
	
	\subsection{Tractability of the proposed approach}
	Problem PP$_\veps$ contains a mixed-integer robust constraint in $\Fcal_{\text{PP}_{\veps}}$ as defined in \eqref{eq:feas_set_pp}.
	Such constraint can be solved by a branch-and-bound algorithm whenever the robust constraint 
	\begin{align}\label{eq:constr_recap}
		g_h(x,\theta)\leq 0, \forall\theta\in\Dcalj
	\end{align}
	admits an equivalent deterministic reformulation, for $ j\in\{1,...,K\}$ and $ h\in\{1,...,Z\}$. 
	In the following we fix, for simplicity, two indices $j$ and $h$. In several cases, an equivalent deterministic reformulation for \eqref{eq:constr_recap} exists, which is the case, e.g., when the function $g_h$ is affine in $\theta$ and the set $\Dcalj$, is compact and convex (see \cite{bertsimas2022robust}, Chapter 2 and 5). More generally, the robust constraint appearing in \eqref{eq:constr_recap} can be solved whenever the function $g_h$ is convex or monotone with respect to $\theta$, and $\Dcalj$ is a polytope. Then, according to Bauer's Maximum Principle \cite{bauer1958minimalstellen}, the maximum with respect to $\theta$ is attained at least at one of the vertices of the uncertainty set. Note that, although this induces an equivalent deterministic reformulation of \eqref{eq:constr_recap}, vertex enumeration can be computationally expensive for high-dimensional sets $\Dcalj$. In addition, there are some specific classes of nonlinear functions, e.g., concave with respect to the uncertainty, for which an equivalent tractable reformulation can be derived (see \cite{bertsimas2022robust}, Chapter 13). Last, we also mention the case in which $g_h$ is $L$-Lipschitz with respect to $\theta$, uniformly in $x$. Constraint tightening based on Lipschitz bounds is possible (see e.g.\ \cite{kohler2021computationally}), but these bounds are typically very conservative unless the size of the set $\Dcalj$ is small enough.
	
	Last, we observe that if only feasibility is of interest, the computational complexity of PP$_\veps$ can be reduced. Indeed, in $\Fcal_{\text{PP}_{\veps}}$, defined as in \eqref{eq:feas_set_pp}, we optimally choose for which of the sets in $\{\Dcalj\}_{j=1}^K$ the constraint must be robustly satisfied, which is a combinatorial problem. An alternative would be to select such sets greedily, e.g.,  by ordering $\{\tpj\}_{j=1}^K$ in decreasing order, selecting the first elements whose sum is at least $1-\veps$, and enforcing robust constraint satisfaction with respect to the corresponding uncertainty sets. This idea essentially fixes the value of the indicator function in $\Fcal_{\text{PP}_{\veps}}$ to some predefined values, thus it might significantly reduce the computational complexity, but it may yield a suboptimal solution. Note, however, that the logical ``or'' in $\Fcal_{\text{PP}_{\veps}}$ is still a combinatorial constraint, which, however, depends on the problem of interest (e.g., control of dynamical systems with logical constraints), rather than on the proposed approach.

	\subsection{Application: Optimal control of PWA systems}\label{subsec:pwa_details}
	The class of constraint functions described in \eqref{eq:g_linear} is often encountered in several optimization problems of practical interest, i.e., optimization problems with a class of logical constraints. In this section, we show how to implement our algorithm for MPC problems for discrete-time PWA systems with chance constraints. PWA systems are a special class of nonlinear system, where the function describing the dynamics is indeed piecewise affine in  state, input, and uncertainty. This class of systems includes linear systems, and they are often used to model physical systems with logical or discrete components (e.g., energy storage systems), or to approximate nonlinear systems with a user-defined level of accuracy \cite{gharavi2024efficient}.
	
	\subsubsection{Modeling}
	The dynamics of a general PWA system with additive uncertainty are described by
	\begin{align}\label{eq:pwa_sys}
		s_{t+1} = A_l s_t + B_l u_t + v_l + C_l \eta_t \ \Leftrightarrow \  
		\begin{bmatrix}
			s_t \\ u_t \\ \eta_t
		\end{bmatrix}
		\in\Omega_l,
	\end{align}
	where $A_l, B_l, C_l$, and $v_l$ are the system matrices and vectors of appropriate dimension, for $l = 1,...,m$, with $m\in\Z_{>0}$. Also, $s_t\in\R^{n_s}$ and $u_t\in\R^{n_u}$ are, respectively, the state and input of the system at time step $t$, and $\Omega_l$, $l = 1,...,m$, are non-empty polytopic sets such that $\{\Omega_l\}_{l=1}^m$ is a polytopic partition of the state, input, and uncertainty space.
	The variable $\eta_t\in\R^{n_\eta}$ is an exogenous additive stochastic signal, and it is described by a distribution inducing the probability measure $\P$.
	Given a stage cost function $\ell:\R^{n_s}\times\R^{n_u}\rightarrow\R_{\geq0}$ and a prediction horizon $\horizon\in\Z_{>0}$, the goal of the controller is to employ the system dynamics \eqref{eq:pwa_sys} to find a sequence of inputs $\us_{t},...,\us_{t+N_\text{pred}-1}$ such that the expected performance of the system over the prediction window 
	\begin{align*} 
		\E\textsubscript{$\eta_{t},...,\eta_{t+N_\text{pred}-1}$} \left[\sum_{k=1}^{N_\text{pred}} \ell(s_{t+k}, u_{t+k-1}) \right] 
	\end{align*}
	is minimized, while satisfying the following constraints on the state and on the input of the system:
	\begin{align*} 
			& \Pr(s_{t+k}\in\Scal, \ k=1,...,N_\text{pred}) \geq 1-\varepsilon, 
			\\ & u_{t+k} \in \Ucal,  \ k=0,...,N_\text{pred}-1,
	\end{align*}
	where $\Scal\subset\R^{n_s}, \Ucal\subset\R^{n_u}$ are polytopes.
	
	Now, to match the constraint formulation in \eqref{eq:g_linear}, we introduce the following compact notation: $\ss_t = [s_t^\top \ ... \ s_{t+\horizon}^\top]^\top, \uu_t = [u_t^\top \ ... \ u_{t+\horizon-1}^\top]^\top, \eeta_t = [\eta_t^\top \ ... \ \eta_{t+\horizon-1}^\top]^\top$, and
	$$\FF_h = \begin{bmatrix}
		I \\ A_{h_0} \\ A_{h_1} A_{h_0} \\ \vdots \\ \Pi_{k=0}^{N_\text{pred}-1} A_{h_k} 
	\end{bmatrix},
	\ \ \vv_h = \begin{bmatrix}
		v_{h_0} \\ v_{h_1} \\\vdots \\ v_{h_{N_\text{pred}-1}}
	\end{bmatrix}
	$$
	$$
	\GG_h = \begin{bmatrix}
		0_{n_s\times n_u} & 0_{n_s\times n_u} & \ldots & 0_{n_s\times n_u}
		\\
		B_{h_0} & 0_{n_s\times n_u} & \ldots & 0_{n_s\times n_u}
		\\ A_{h_1} B_{h_0} & B_{h_1} & \ddots & \vdots
		\\ \vdots  & \ddots & \ddots & 0_{n_s \times n_u}
		\\ \Pi_{k=1}^{N_\text{pred}-1} A_{h_k} B_{h_0} & \Pi_{k=2}^{N_\text{pred}-1} A_{h_k} B_{h_1} & ... & B_{h_{\horizon-1}}
	\end{bmatrix}
	$$
	$$ \GGamma_h = \begin{bmatrix}
		0_{n_s\times n_{\eta}} & 0_{n_s\times n_{\eta}} & \ldots & 0_{n_s\times n_{\eta}}
		\\
		C_{h_0} & 0_{n_s\times n_{\eta}} & \ldots & 0_{n_s\times n_{\eta}}
		\\ A_{h_1}C_{h_0} & C_{h_1} & \ddots & \vdots
		\\ \vdots  & \ddots & \ddots & 0_{n_s\times n_{\eta}}
		\\ \Pi_{k=1}^{N_\text{pred}-1} A_{h_k} C_{h_0}  & \Pi_{k=2}^{N_\text{pred}-1} A_{h_k} C_{h_1} & ... & C_{h_{\horizon-1}}
	\end{bmatrix}$$
	with $h\in\{1,...,m^{N_\text{pred}-1}\}$.
	Then, the system dynamics \eqref{eq:pwa_sys} over the horizon are:
	\begin{align}\label{eq:compact_dyn_1} 
		\begin{split}
			&\ss_{t} =   \FF_h s_t + \GG_h \uu_t + \vv_h + \GGamma_h \eeta_t 
			\\& \Leftrightarrow s_t\in\Omega_{h_0}, ..., s_{t+N_\text{pred}-1}\in\Omega_{h_{N_\text{pred}-1}}.
		\end{split}
	\end{align}
	Let $\LL_u, \boldsymbol{l}_u$ describe the polytopic representation of $\Ucal^{N_\text{pred}}$, i.e., $(u_t, ..., u_{t+N_\text{pred}-1}) \in \Ucal^{N_\text{pred}} \Leftrightarrow \LL_u\uu_t\leq\boldsymbol{l}_u$, and similarly let $\LL_{s, h}, \boldsymbol{l}_{s, h}$ describe the polytopic representation of $(\Omega_{h_0}\cap\Scal) \times ... \times (\Omega_{h_{N_\text{pred}-1}}\cap\Scal)$, the final chance constrained optimal control problem can be written as
	\begin{align}\label{ocp}
		\begin{split}
			\min_{\uu_t} \ &  \E \left[\Lscr(s_t, \uu_t, \eeta_t)\right]
			\\ \text{s.t.} \ & \Pr(\lor_{h=1}^Z (\LL_{s, h} \ss_{t} \leq \boldsymbol{l}_{s, h})) \geq 1-\veps 
			\\& \LL_u\uu_t\leq\boldsymbol{l}_u,
		\end{split}
	\end{align}
	with $\Lscr(s_t, \uu_t, \eeta_t) :=\sum_{k=1}^{\horizon} \ell(s_{t+k}(s_t, \uu_t, \eeta_t), u_{t+k-1})$, and where $s_{t+k}(s_t, \uu_t, \eeta_t)$ indicates that the predicted state $s_{t+k}$ is a function of the current state, the applied input sequence, and the uncertainty, according to \eqref{eq:compact_dyn_1} (thus, by explicitly substituting the PWA dynamics in the cost function).
	Note that \eqref{ocp} matches the formulation of CP$_\veps$ with $x=\uu_t, \Xcal=\{\uu: \LL_u\uu\leq\boldsymbol{l}_u\}, \theta=\eeta_t, Z=M^{N_\text{pred}-1}, J(x, \theta)=\Lscr(s_t, \uu_t, \eeta_t)$, and $g_h(x, \theta) = \LL_{s, h} \ss_{t} - \boldsymbol{l}_{s, h}$, for $h\in\{1,...,Z\}$. In particular, each function $g_h$ is linear, thus matching the special case considered in \eqref{eq:g_linear}. However, despite the linearity of $g_h$, optimal control problems for PWA systems are non-convex also in the deterministic case (due to the ``or'' constraint, which is intrinsic of the PWA dynamics), and typically they are formulated as mixed-integer problems \cite{bemporad1999control, heemels2001equivalence}.

	\subsubsection{Constraint modifications and Lipschitz constants}
	
		In order to apply the proposed approach to MPC problems for PWA systems, we need some ingredients. The tightened and robust constraints are of the type discussed in \eqref{eq:robust_linear}, \eqref{eq:relaxed_linear}, and the tightening and relaxation parameters can be computed as in \eqref{eq:tight_linear}, \eqref{eq:gamma_linear}. 
		Note that, in the case of PWA systems, in the resulting problem PP$_{\veps-\delta}$ we require that all the uncertainty values in a set $\Dcalj$ activate the same mode sequence (since a given sequence of control and uncertainty uniquely determines the mode sequence). Thus, we may require the sets in the uncertainty partition to have non-overlapping boundaries, to allow that different mode sequences can be activated by sets whose boundaries may have non-empty intersections. This, in practical implementation, can be achieved by subtracting an arbitrarily small positive term to \eqref{eq:tight_linear} (and adding it to \eqref{eq:gamma_linear}).
		
		For the performance bounds, the Lipschitz constants of the PWA dynamics with respect to $\uu_t$ and $\eeta_t$ are needed. Note that generic vector-valued continuous PWA functions (thus also the ones in \eqref{eq:compact_dyn_1}) are automatically Lipschitz, with constant equal to the largest induced matrix norm defining the slope (this is evident in the scalar case, and can be derived by using the triangle inequality for the multi-dimensional case). Then, if the stage cost is also Lipschitz, the function $\Lscr$ will be Lipschitz due to composition of Lipschitz functions. 
		
		For example, suppose the stage cost is the 1-norm, which is commonly used for control of PWA systems (see e.g.\ \cite{he2024approximatea}, as the resulting problem can be then solved as a mixed-integer linear program). Thus, by considering $Q\in\R^{n_s\times n_s}$ and $R\in\R^{n_u\times n_u}$ with non-negative entries, the stage cost is 
		$$\ell(s,u) = \|Qs\|_1 + \|Ru\|_1.$$
		Since any norm is 1-Lipschitz (in view of the reverse triangle inequality), the Lipschitz constants of the cost function with respect to $\eeta_t$ and with respect to $\uu_t$ are, respectively
		\begin{align}\label{eq:lip_pwa}
			\begin{split}
				&L_\eta = \|Q\|_1\max_{h\in\{1,...,Z\}} \|\GGamma_h\|_1, 
				\\&  L_u = \left(\|Q\|_1\max_{h\in\{1,...,Z\}}\|\GG_h\|_1 + \|R\|_1 \right)\sqrt{\horizon n_u},
			\end{split}
		\end{align}
		where we have leveraged the composition of the PWA dynamics with the 1-norm, i.e.: $\|Q\|_1$ is the Lipschitz constant of the map $s\rightarrow\|Qs\|_1$, $\|R\|_1$ is the Lipschitz constant of the map $u\rightarrow\|Ru\|_1$, and $\max_{h\in\{1,...,Z\}}\|\GGamma_h\|_1$ and $ \max_{h\in\{1,...,Z\}}\|\GG_h\|_1$ are the Lipschitz constants of the PWA dynamics \eqref{eq:compact_dyn_1} with respect $\eeta_t$ and $\uu_t$, respectively. Last, $L_u$ is rescaled in accordance to the equivalence between 1-norm and 2-norm in view of Assumption \ref{ass:Jlip_1}.

	\subsubsection{Uncertainty partitioning}
	
	So far, a partition that satisfies Assumption \ref{ass:partition} was assumed to be given. We now discuss three options to derive such a partition. 
	
	\paragraph{Clustering}\label{par:clus}
	A partition of $\Dcal$ may be derived from a clustering algorithm, followed, e.g., by a Voronoi partition \cite{fortune2017voronoi} based on the resulting centroids. This method has the benefit of capturing a pattern in the uncertainty (as explored in \cite{cordiano2024scenario}), but it requires additional data to perform the clustering algorithm. Also, modern implementations for clustering and Voronoi partitioning allow for iterative algorithms with polynomial complexity, and in practice can be very fast. 
	
	\paragraph{Iterative gridding}\label{par:grid}
	Second, if no data to perform clustering is available, or if the partitioning has to be computed fast, an idea could be to just partition the uncertainty domain such that the sets $\{\Dcalj\}_{j=1}^K$ are with a convenient shape, e.g.\ box uncertainty sets. 
	That is, assume an initial box that contains the domain $\Theta$ is available: $\Dcal:=\{\theta\in\R^{n_\theta}: \underline{\rho}_i \leq \theta_i \leq \Bar{\rho}_i, \forall i\in\{1,...,n_\theta\} \}$. Then we can find the entry $i$ for which the interval $[\underline{\rho}_i, \Bar{\rho}_i ]$ is the largest, and split it into half. This procedure can be repeated iteratively, by splitting into half the largest interval among the sets currently present in the partition, until a maximum number $K$ is reached.
	
	\paragraph{Adaptive splitting}\label{par:ada}
	Third, a variant of the previous idea can be tailored based on the receding-horizon nature of MPC. Note that instead of considering all the entries of the uncertainty in this gridding procedure, an interesting variant consists of partitioning the dimensions of $\eeta_t$ that may push the state of the system close to the constraint boundaries. For example, from the predicted state trajectories at time step $t$ we can infer at which time step $\Bar{k}\in\{1,...,\horizon\}$ the predicted state $x_{t+\Bar{k}}$ is closest to the constraint boundaries. Then the idea could be to derive a partition for $\eta_{t+\Bar{k}}$ only (e.g.\ by using the idea in Section \ref{par:grid} by starting from a box containing $\eta_{t+\Bar{k}}$) for the solution of the MPC problem at the next time step, with the intuition that having a refined partition on the uncertainty elements that may cause a constraint violation is likely to lead significant performance improvements (since high system performance is often achieved in the vicinity of the constraint boundaries).

	\section{Numerical example}\label{sec:example}
	
	\subsection{Setup}
	The simulation results in this section are based on the following PWA system: 
	\begin{equation}\label{eq:pwa_example}
		\begin{gathered}
			A_1 = \begin{bmatrix}
				0.8 & 1 & 1 \\
				0 & 0.9 & 1 \\
				0 & 0 & 0.2
			\end{bmatrix}, \ 
			A_2 = \begin{bmatrix}
				0.8 & 1 & 1 \\
				0 & -0.9 & 1 \\
				0 & 0 & -0.2
			\end{bmatrix} \\
			A_3 = \begin{bmatrix}
				0.8 & 1 & 1 \\
				0 & 0.5 & 1 \\
				0 & 0 & 0.5
			\end{bmatrix}, 
			B = \begin{bmatrix}
				0 \\
				0 \\
				1
			\end{bmatrix}, \ 
			C = \begin{bmatrix}
				0 & 0 \\
				1 & 0 \\
				0 & 1
			\end{bmatrix},
		\end{gathered}
	\end{equation}
	with regions defined by
	\begin{equation*}
		\begin{gathered}
			\Omega_1 = \{s\in\R^3: s_1 \leq -1\}, \ \Omega_2 = \{s\in\R^3: -1 \leq s_1 \leq 1 \},
			\\ \Omega_1 = \{s\in\R^3: s_1 \geq 1\}.
		\end{gathered}
	\end{equation*}
	Note that the PWA dynamics described by \eqref{eq:pwa_sys} and \eqref{eq:pwa_example} are a continuous function of $s_t, u_t, \eta_t$, by simply choosing $v_l = 0, l\in\{1, 2, 3\}$. In all the MPC problems that we are going to solve, we consider an initial state $x_0=[1.5 \ 2 \ 1]^\top$, $\horizon=5$, $Q=2I_3$ (where $I_3$ is the 3-dimensional identity matrix), $R=1$, and we require that the system satisfies the following constraints:
	\begin{align*}
		& \|u_{t+k}\|_\infty \leq 0.7, k=0,...,\horizon-1,
		\\& \Pr(|s_{3, t+k}| \leq 0.7, k=1,...,\horizon) \geq 1-\veps,
	\end{align*}
	with $\veps = 0.15$. At each time step $t\in\Z_{\geq0}$, a forecast for the values of the two-dimensional disturbance is available from the following model:
	\begin{align}\label{eq:unc}
		\eta_{t+k,i} = a_{k, i} \sin(\omega_i k + \phi_i) + w_{k, i}
	\end{align}
	with $i\in\{1,2\}$ and $k\in\{0,...,\horizon-1\}$.
	We consider that the given source of uncertainty exhibits more variability as the time grows, over the prediction window. Thus we consider $a_{k,1}$ and $a_{k,2}$ to be uniformly distributed, respectively over the intervals $[0.02(k+1), 0.03(k+1)], [0.04(k+1), 0.06(k+1)], k=0, ..., N-1$. Similarly, $w_{k, i}$ is uniformly distributed over  $[-0.03(k+1), 0.03(k+1)], k\in\{0, ..., N-1\}$, $i\in\{1,2\}$.
	Then, $\omega_1$ and $\omega_2$ follow a mixture of Gaussian distributions with means in the sets $\{0.05, 0.12, 0.3, 0.5, 0.75\}$ and $\{0.1, 0.24, 0.6, 1, 1.5\}$ respectively, variances equal to 0.01, and each component has weight $\{0.05, 0.1, 0.4, 0.4, 0.05\}$ for both $i=1,2$. Last, $\phi_1$ and $\phi_2$ are uniformly distributed in $[-0.1, 0.1]$.
	For the following simulations, we initially assume the domain $\Omega$ of the distribution to be contained in a hyper-rectangle $\Dcal$, i.e.
	\begin{align}\label{eq:domain}
		\underline{\rho} \leq \eeta_{t} \leq \Bar{\rho}, \ \forall \eeta_t\in\Omega,
	\end{align}
	where the inequalities are intended element-wise, and $\underline{\rho}, \Bar{\rho}$ are known quantities given by
	\begin{equation*}
		\begin{gathered}
			\Bar{\rho} = [0.06 \ 0.09 \ 0.12 \ 0.18 \ 0.18 \ 0.27 \ 0.24 \ 0.36 \ 0.3 \ 0.45]^\top
			\\ \underline{\rho} = -\Bar{\rho}.
		\end{gathered}
	\end{equation*}
	Next, we propose two main experiments to test the proposed approach. The first experiment validates our theoretical results in terms of feasibility and performance. In particular, we will evaluate how the computational complexity and the conservatism of the obtained solution can be traded off by adjusting $K$, and we will compare the obtained results with randomized algorithms \cite{campi2008exact, mohajerinesfahani2015performance}. In the second experiment, we test the proposed approach in a closed-loop MPC problem, by comparing two different approaches for partitioning.
	For all the simulations, we use Gurobi 10.0.2 as a solver on a Macbook Pro (Apple Silicon M1 pro, 32GB RAM)\footnote{Code available at: https://github.com/fracordi/partition-control-pwa}. 

	\subsection{Open-loop validation}
	In the first simulation of this experiment we validate our theoretical findings. In particular, concerning feasibility we evaluate how the observed constraint violation changes for a growing sample size, while concerning performance we will test how the performance bounds computed by means of Theorem \ref{th:optimal_values} change for a growing $K$.
	
	Figure \ref{fig:confidence} shows the feasibility results for a growing sample size and different choices for $K, \delta$.  More precisely, for each value of $N$ on the $x$-axis, and for a certain choice $(K,\delta)\in\{(5,0.1), (20, 0.05), (100, 0.01)\}$, we partition the uncertainty domain \eqref{eq:domain} in $K$ sets using the gridding method described in Section \ref{par:grid}, and we compute $\{\hthetaj\}_{j=1}^K, \{\hpj\}_{j=1}^K$, respectively, as in \eqref{eq:hthetaj}, \eqref{eq:hpj}. We solve the partition-based
	problem derived from \eqref{ocp}, we apply the resulting input sequence to the system, and measure the constraint violation using a validation set (of size sufficiently large for a reliable estimate, i.e.\ $N_\text{sample}=10^5$). The whole procedure is then repeated 500 times (for each value of $N, K, \delta$), in order to compute the mean value of the expected constraint violation (solid line) and a confidence interval of one standard deviation (shaded area).

	\begin{figure}
		\centering
		\includegraphics[scale=0.52]{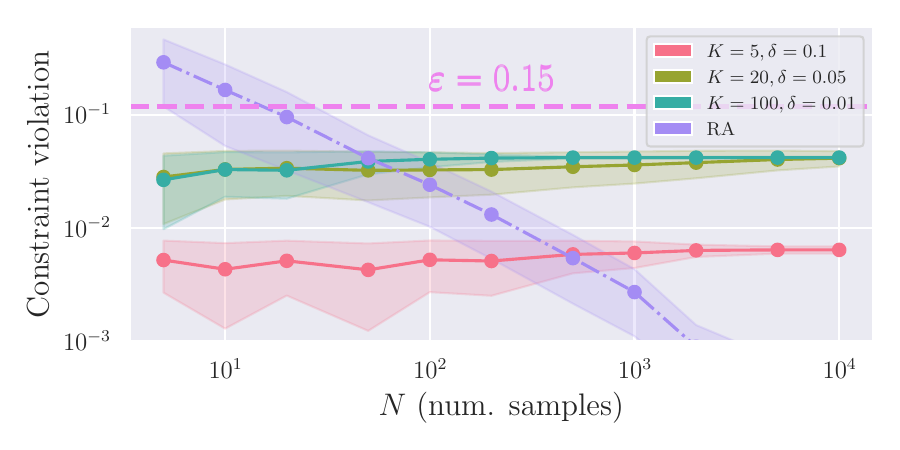}
		\caption{Empirical constraint violation for different choices of $K,\delta$, and compared with randomized algorithms (RA). For each case, the bold line is the mean value, whereas the shaded area is a confidence interval of one standard deviation, both computed empirically over 500 simulations.}
		\label{fig:confidence}
	\end{figure}

	We observe that the results confirm the theoretical findings, i.e., for $N$ large enough, PP$_{\veps-\delta}$ is guaranteed to provide a feasible solution for CP$_\veps$. In addition, for larger $K$ or smaller $\delta$, the observed constraint violation approaches the theoretical value $\veps=0.15$, i.e., the solved problem becomes less conservative (since in such cases PP$_{\veps-\delta}$ better approximates the original CP$_\veps$), while preserving feasibility. However, with a given $N$, we also observe more variability in the measured violation for larger values of $K$, 
	or smaller values for $\delta$, compared to $(K,\delta)=(5, 0.1)$ (note the logarithmic scale on the $y$-axis). This is indeed expected since, in view of Lemma \ref{lemma:uniform}, the confidence of the accuracy in the estimation of $\{\hpj\}_{j=1}^K$ might decrease for large $K$ or small $\delta$.
	However, in this specific case, selecting $K$ very high or $\delta$ very small may not necessarily give significant improvements in the performance, compared to $K=20, \delta=0.05$.

	The approach is then compared with randomized algorithms \cite{campi2008exact, mohajerinesfahani2015performance}, denoted by RA, where $N$ sampled constraints are employed in the optimization problem in place of the chance constraint. Although more sample efficient, RA can be more conservative, in the sense that the observed constraint violation can be significantly smaller than $\varepsilon$ (in Figure \ref{fig:confidence} the constraint violation is smaller than 0.001 for $N\geq2000$, and larger values of $N$ are not tested in RA due to an excessive computation time).
	Indeed, in RA the $N$ samples originate $N$ distinct hard constraints, which certainly reduce the risk of constraint violation, but this comes at the expense of an excessive conservatism. In contrast, in our approach, the estimate  $\{\hpj\}_{j=1}^K$ becomes more and more accurate with growing $N$, without introducing conservatism in the system performance.
	
	Table \ref{table:performance} shows the performance bounds computed by solving PP$_{\veps-\delta}$ and RP$_{\veps+\delta}$, and by applying Theorem \ref{th:optimal_values}, where the computation of the Lipschitz constants \eqref{eq:lip_pwa} yields
	$$L_\eta = 34.64, \ L_u = 79.69.$$
	Together with the performance bounds, also the total solver time, to solve both PP$_{\veps-\delta}$ and RP$_{\veps+\delta}$, are reported in Table \ref{table:performance}. This is done for $(K,\delta)\in\{5,10,20,50,100,200\}\times\{0.05, 0.01\}$, and for each choice of $(K,\delta)$, $N$ is computed according to Assumption \ref{ass:uniform:M}, with $\beta=10^{-4}$. As in the previous simulation, each problem is solved 500 times, and the average of the obtained bounds and of the solver time is listed in Table \ref{table:performance}. As expected, it is possible to note that by increasing $K$ the bounds become indeed less conservative, since this would decrease the error due to the partitioning, i.e.\ the term $c_2$ in \eqref{eq:c1,c2}. The same happens when decreasing $\delta$, which is also expected since, according to Assumption \ref{ass:uniform:M}, this requires more samples, thus decreasing the sampling error $c_1$ in \eqref{eq:c1,c2}. 
	
	Note, however, that the level of conservatism of these bounds is problem-specific, and they are largely affected by the dependency on the Lipschitz constants $L_u, L_\eta$ of the system. The considered system is PWA, thus nonlinear, and the Lipschitz constants computed with \eqref{eq:lip_pwa} are conservative. An alternative would be to compute them via gridding or sampling. Although this would result in an approximation, the obtained constants are
	$L_{\eta,\text{approx}} \approx 17, \ L_{u,\text{approx}} \approx 60$,
	thus much less conservative than the theoretical ones computed above. Using such approximated constants, the obtained bounds with, e.g., $K=20, \delta=0.05$, lead to 
	$$J^\star_{\text{CP}_{\veps}} \in [54, 89], $$
	which are less conservative compared to the bounds in Table \ref{table:performance}. In Table \ref{table:performance} we also show the total time required to solve both PP$_{\veps-\delta}$ and RP$_{\veps+\delta}$. This grows indeed with $K$, but the computational burden remains affordable for $K=5-20$, for systems with a large enough sampling time, e.g., building heating systems \cite{pippia2021scenariobased}. Note that, in general, PWA systems are computationally demanding even in the nominal case, and, with the current state-of-the-art solvers, the computational complexity can grow exponentially with the prediction horizon. In general, if theoretical performance bounds are not of interest, but only a feasible solution is desired, gradient-free methods are viable options to provide solutions to PP$_\veps$, although potentially suboptimal, without resorting to mixed-integer reformulations.

	\begin{table}
		\begin{tabular}{|c|ccc|ccc|}
			\hline
			& \multicolumn{3}{c|}{$\delta=0.05$} 
			& \multicolumn{3}{c|}{$\delta=0.01$} \\ \cline{2-7}
			$K$ & LB    & UB     & Time tot.   &  LB    & UB      & Time tot. \\ \hline
			5   & 35.7  & 109.9    &    2.2        &  47.1  &  98.5   &  2.3 \\
			10  & 38.9  &  106.8 &    5.2     &  48.9  &  96.8   &  5.4 \\
			20  & 41.3  & 102.7  &    12.2    &  49.7  &  94.2   &  12.6 \\
			50  &  45.5 & 98.6   &    29.4     &  51.6  &  92.4   &  31.6 \\
			100 &  48.4 &  95.7  &    54.2     &  52.9  &  91.1   &   58.7\\
			200 & 50.2  & 93.9   &    126.5    &  54.2  &  90.6   &   146.9\\ 
			\hline
		\end{tabular}
		\caption{Avg. lower and upper performance bounds (LB, UB) and total solver time (in seconds), for different choices of $K, \delta$.}
		\label{table:performance}
	\end{table}

	\subsection{Closed-loop experiment}
	In this section, we compare two different partitioning strategies in a closed-loop experiment, where, at each time step $t\in\{0,...,T_{\text{cl}}\}$, we solve the partition-based problem PP$_{\veps-\delta}$ for system \eqref{eq:pwa_example}. Specifically, we compare the
	partitioning strategy defined in Sections \ref{par:clus}, i.e., $K$-means clustering, and \ref{par:ada}, i.e., adaptive splitting. In this experiment, we evaluate the closed-loop cost
	$$\ell_{\text{cl}}(t):=\ell(x_t, u_{t-1}) = \|Qx_t\|_1+\|Ru_{t-1}\|_1,$$
	$\forall t\in\{1,...,T_\text{cl}\}$, with $T_\text{cl}=80$, which is computed by repeating this closed-loop experiment 500 times, in order to compute the mean value and a confidence interval of one standard deviation. Note that when the average stage cost is very small, subtracting one standard deviation may lead to a negative value in the plot, which corresponds to $-\infty$ in a logarithmic scale. Thus we limit the $y$-axis to 0.1. In particular, Figure \ref{fig:cl_loop} shows the results for the two partitioning options (where ``ADPT'' denotes the approach discussed in Section \ref{par:ada}, and ``KMNS'' the one in Section \ref{par:clus}), and for $K\in\{4,8\}$. The confidence parameters are set to $\delta=0.05, \beta=10^{-4}$, and $N$ samples are drawn from \eqref{eq:unc} according to Assumption \ref{ass:uniform:M}.
		
	\begin{figure}
		\centering
		\includegraphics[scale=0.55]{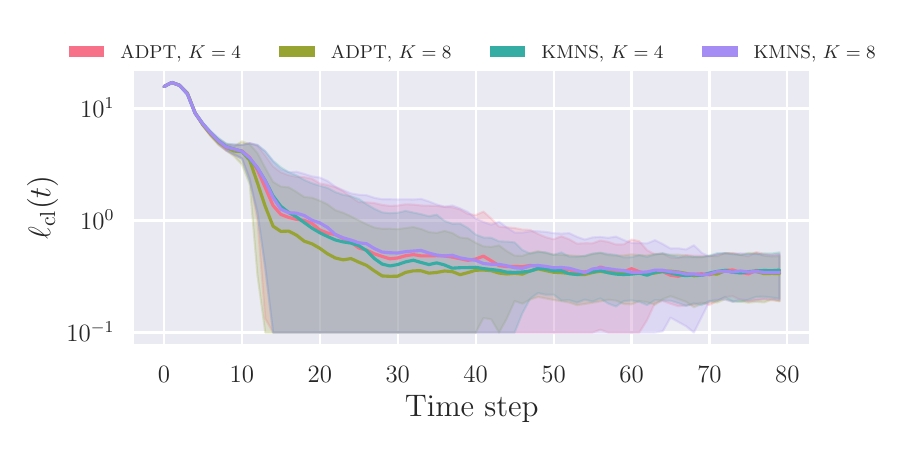}
		\caption{Closed-loop cost for the adaptive splitting (ADPT), and $K$-means (KMNS) with Voronoi partitioning. In both cases we test $K=4$ and $K=8$. For each case, the bold line is the mean value, whereas the shaded area is a confidence interval of one standard deviation, both computed empirically over 500 simulations.}
		\label{fig:cl_loop}
	\end{figure}
	
	As we can see in Figure \ref{fig:cl_loop}, the choice $K=4$ leads to substantially similar performance for both ADPT and KMNS. However, increasing the number of clusters to $K=8$ does not significantly alter the performance for KMNS; in contrast, ADPT exhibits significantly better performance, both in terms of speed of convergence, and in terms of variability of the closed-loop cost.
	This may indicate that, in control applications, having a partition for the uncertainty regions that may cause a constraint violation is likely to lead to significant performance improvements, since high system performance is often achieved in the vicinity of the constraint boundaries (which is the idea behind ADPT). Conversely, clustering-based partitions, i.e., KMNS, have the main feature of detecting a pattern in the data, but the effectiveness of the algorithm highly depends on the chosen $K$, and they do not take into account the effect of the uncertainty on the optimization problem. Last, the time required by the two algorithms to perform the partition in, on average over the 500 experiment and over time, is, in milliseconds: 0.25 (ADPT, $K=4$), 0.48 (ADPT, $K=8$),  3.5 (KMNS, $K=4$), 4.7 (KMNS, $K=8$). ADPT is faster than KMNS, since clustering algorithms typically need several iterations to converge. However, we see that they can be both considered viable options for online control algorithms, as the required partitioning time is negligible compared to the solver time (1-3 seconds depending on $K$).

	\section{Conclusions}\label{sec:conc}
	This paper has presented a new strategy for solving chance-constrained optimization problems. The proposed approach is based on partitioning the uncertainty domain in a user-defined number of sets, allowing for a trade off between conservatism and computational complexity. Rigorous feasibility and performance analyses have been provided, and the results have been tailored for a class of optimization problems that includes optimal control of dynamical systems with probabilistic logical constraints. Finally, the effectiveness of the approach has been tested on a model predictive control problem for PWA systems. 
	
	Ongoing and future work consists of making the approach more sample-efficient, as well as exploring methods to speed up the solution of MPC problems for chance-constrained PWA systems.

	\section*{Appendix}

	\subsection{Proof of Lemma \ref{lemma:uniform} }
	Lemma \ref{lemma:uniform} follows by observing that, under Assumption \ref{ass:partition}, the variable 
	$Y^{(j)} = \sum_{i=1}^M \mathbf{1}_{\thetai\in\Dcalj}$
	follows a multinomial distribution with parameter $\pj$, $\forall j\in\{1,...,K\}$. Then the result follows by applying the Bretagnolle–Huber–Carol inequality (\cite{vandervaart2013weak} Section A.6.6), on the random vector $[Y^{(1)}\ ...\ Y^{(K)}]^\top$, which, under Assumptions \ref{ass:iid} and \ref{ass:uniform:M}, yields:
	$$\P^N \left(\sum_{j=1}^K |Y^{(j)} - N \pj| \geq 2\sqrt{N} \delta\right) \leq 2^K e^{-2\delta^2},$$
	and, in view of the Scheffe's identity \cite{devroye2001combinatorial}, it leads to:
	$$\P^N \left(\max_{\Jcal\subseteq\{1,...,K\}} \left|\sum_{j\in\Jcal}^K \left(\hpj - \pj\right) \right| \geq \frac{\delta}{\sqrt{N}} \right) \leq 2^K e^{-2\delta^2},$$
	which is equivalent to the claim.

	\subsection{Proof of Lemma \ref{lemma:inclusion} }
	In this proof, we consider, for completeness, the general case in which the constraint functions $g_h$ are multidimensional, i.e.: $g_h:\Xcal\times\Dcal\rightarrow\R^{n_{\text{c},h}} $ (i.e., $g_h$ represents a number $n_{\text{c},h}$ of constraints that have to be satisfied). Also the functions $\gammaj_h$ are defined accordingly.
	Let, for short
	$\pj := \Pr(\theta\in\Dcalj).$
	From the law of total probability, and from Assumption \ref{ass:partition}, we have
	\begin{align}\label{app:inclusion:feas}
		\begin{split}
			& x\in\Fcal_{\text{CP}_\veps} 
			\\ & \Leftrightarrow \Pr(g(x,\theta)\leq 0)\geq 1-\veps
			\\& \Leftrightarrow \sum_{j=1}^{K} \pj \Pr(\lor_{h=1}^Z  g_h(x,\theta)\leq 0|\theta\in\Dcalj) \geq 1-\veps
		\end{split}
	\end{align}
	where the last one is in view of \eqref{eq:cc}.
	Now, we claim that the following relation holds for any $x\in\Xcal$:
	\begin{align}
		\begin{split}\label{app:inclusion:bound}
			&\Pr(\lor_{h=1}^Z  g_h(x,\theta)\leq 0|\theta\in\Dcalj) 
			\\& \leq \mathbf{1}_{[\lor_{h=1}^Z( g_h(x,\hthetaj)\leq \gammaj_h(x,\hthetaj))]}, 
		\end{split}
	\end{align}
	for all $j\in\{1,...,K\}$. If the value of the indicator function in the right-hand side of \eqref{app:inclusion:bound} is 1, then the relation is trivial. Thus, in the following we consider any $\Bar{x}\in\Xcal$ such that value of the indicator function is 0. This is equivalent to
	\begin{align*}
		g_h(\Bar{x},\hthetaj) & \not\leq \gammaj_h(\Bar{x},\hthetaj), \ \ \forall h\in\{1,...,Z\},
	\end{align*}
	which, in view of \eqref{eq:gamma}, implies that for all $h\in\{1,...,Z\}$ there exists an entry $l\in\{1,...,n_{\text{c}, h}\}$ such that we have
	\begin{align*}
		g_{h,l}(\Bar{x},\hthetaj) & > \gammaj_{h, l}(\Bar{x},\hthetaj)
		\\&  > g_{h,l}(\Bar{x},\hthetaj) - g_{h,l}(\Bar{x},\theta),  \forall \theta\in\Dcalj,
	\end{align*}
	which allows to say that for all $h\in\{1,...,Z\}$ there exists $l\in\{1,...,n_{\text{c}, h}\}$ such that
	$ g_{h,l}(\Bar{x},\theta) > 0,  \forall \theta\in\Dcalj.$
	This can be equivalently written as:
	$$\Pr(\exists h\in\{1,...,Z\}: g_h(\Bar{x},\theta)\leq 0|\theta\in\Dcalj) = 0,$$
	which proves \eqref{app:inclusion:bound}.
	Now, the result follows from similar steps in the proof of Theorem \ref{th:feas_prob}. Under Assumptions \ref{ass:iid} and \ref{ass:uniform:M}, the following hold true uniformly for all $x\in\Xcal$, with probability at least $1-\beta$:
	\begin{align*}
		1-\veps &\leq \sum_{j=1}^{K} \pj\Pr((\lor_{h=1}^Z  g_h(x,\theta)\leq 0)|\theta\in\Dcalj)
		\\& \leq \sum_{j=1}^{K} \pj\mathbf{1}_{[\lor_{h=1}^Z (g_h(x,\hthetaj)\leq \gammaj_h(x,\hthetaj))]}
		\\& \leq \sum_{j=1}^{K} \hpj\mathbf{1}_{[\lor_{h=1}^Z (g_h(x,\hthetaj)\leq \gammaj_h(x,\hthetaj))]} + \delta,
	\end{align*}
	where the second inequality follows by substituting \eqref{app:inclusion:bound} in \eqref{app:inclusion:feas}, and the third one from Lemma \ref{lemma:uniform}.
	This means that the event
	$$\max_{\Jcal\subseteq\{1,...,K\}}
	\left| \sum_{j\in\Jcal} \left(\Pr(\theta\in\Dcalj) - \hpj  \right) \right|
	\leq\delta$$
	implies both the following events:
	\begin{align*}
		& \Pr(g(x,\theta)\leq0) \geq 1-\veps, \ \forall x\in\Fcal_{\text{PP}_{\veps-\delta}}
		\\&	\sum_{j=1}^{K} \hpj\mathbf{1}_{[\lor_{h=1}^Z (g_h(x,\hthetaj)\leq \gammaj_h(x,\hthetaj))]} \geq 1-\veps - \delta, \  
		\forall x\in\Fcal_{\text{CP}_{\veps}},
	\end{align*}
	where the first one comes from the proof of Theorem \ref{th:feas_prob}. Equivalently, this means that
	\begin{align*}
		& \P^N (\Fcal_{\text{PP}_{\veps-\delta}} \subseteq \Fcal_{\text{CP}_{\veps}} \subseteq \Fcal_{\text{RP}_{\veps+\delta}}) 
		\\& \geq \P^N \left(\max_{\Jcal\subseteq\{1,...,K\}}
		\left| \sum_{j\in\Jcal} \left(\Pr(\theta\in\Dcalj) - \hpj  \right) \right|
		\leq\delta\right)
		\\& \geq 1-\beta.
	\end{align*}

	\subsection{Proof of Lemma \ref{lemma:uniform_concentration} }
	Let $R$ be the diameter of $\Xcal$ (i.e., the largest Euclidean distance between two points in $\Xcal$), and der $r\in(0,R]$. Let $N_\text{cover} := \Ncal(r, \Xcal, \|\cdot\|_2)$ be the covering number of $\Xcal$, i.e., minimum number of balls with radius $r$ required to cover a set $\Xcal$.
	Let $\{\Bcal^{(k)}\}_{k=1}^{N_\text{cover}}$ be such balls and $\{\txk\}_{k=1}^{N_\text{cover}}$ be their centers. This means, that for all $x\in\Xcal$, there exists a $k\in\{1,...,N_\text{cover}\}$ such that $\|x-\txk\|_2\leq r$. Then, it holds that
	\begin{align*}
		\begin{split}
			&\max_{x\in\Xcal} \Big|\E [J(x,\theta)] - \frac{1}{N}\sum_{i=1}^N  J(x,\thetai)\Big|
			\\& \leq \max_{k\in\{1,...,N_\text{cover}\}} \max_{x\in\Bcal^{(k)}} \Big|\E [J(x,\theta)] - \frac{1}{N}\sum_{i=1}^N J(x,\thetai)\Big|.
		\end{split}
	\end{align*}
	Thus, for any $k\in\{1,...,N_\text{cover}\}$, and for all $x\in\Bcal^{(k)}$, from the triangle inequality we have 
	\begin{align}\label{proof:optimal_values:monotone}
		\begin{split}
			& \Big|\E [J(x,\theta)] - \frac{1}{N}\sum_{i=1}^N J(x,\thetai)\Big|
			\\  & \leq
			\begin{aligned}[t]
				& \Big|\E [J(x,\theta)] - \E [J(\txk,\theta)]\Big|
				\\ & + \Big|\E [J(\txk,\theta)] - \frac{1}{N}\sum_{i=1}^N J(\txk,\thetai) \Big|
				\\ & + \Big| \frac{1}{N}\sum_{i=1}^N J(\txk,\thetai) - \frac{1}{N}\sum_{i=1}^N J(x,\thetai) \Big|.
			\end{aligned} 
		\end{split}
	\end{align}
	
	For the first term in \eqref{proof:optimal_values:monotone}, we have, by using Jensen's inequality (note that the absolute value is a convex function), the monotonicity of the expectation, and the Lipschitz continuity of $J$ with respect to $x$:
	\begin{align*}
		\Big|\E [J(x,\theta)] - \E[J(\txk,\theta)]\Big| &
		\leq \E \big[|J(x,\theta) - J(\txk,\theta)|\big]
		\\ & \leq L_x \E \big[\|x-\txk\|_2\big]  
		\leq L_x r,
	\end{align*}
	and the same bound can be found for the third term in \eqref{proof:optimal_values:monotone} in the same way. Concerning the second term, since $\{\txk\}_{k=1}^{N_\text{cover}}$ are fixed, we can apply McDiarmid's inequality \cite{devroye2001combinatorial} and a union bound. More specifically, for each $\{\txk\}_{k=1}^{N_\text{cover}}$, consider the function $F_k$ defined by
		$$F_k(\theta^{(1)}, ..., \theta^{(N)}):= \sum_{i=1}^N J(\txk, \thetai).$$
		In view of the Lipschitz continuity of $J$ (Assumption \ref{ass:Jlip_1}), each function $F_k$ satisfies the so-called bounded difference property, i.e.
		\begin{align*}
			& |F_k(\theta^{(1)}, ..., \thetai, ... \theta^{(N)}) - F_k(\theta^{(1)}, ..., \theta^{(m)}, ... \theta^{(N)})|
			\\& = |J(\txk, \thetai) - J(\txk, \theta^{(m)})|
			\\& \leq L_\theta D, \quad \forall \thetai, \theta^{(m)} \in \Dcal,
		\end{align*}
		with $D$ defined as in the statement.
		Thus, under Assumption \ref{ass:iid}, McDiarmid's inequality can be applied. For any $\lambda\geq0$:
		\begin{align*}
			&\P^N\left(|F_k(\theta^{(1)}, ..., \theta^{(N)}) - \E [F_k(\theta^{(1)}, ..., \theta^{(N)})] | \leq \lambda \right) 
			\\& = \P^N \left(\left|\sum_{i=1}^N J(\txk, \thetai) - N \E[J(\txk, \theta)]\right| \leq \lambda \right) 
			\\& \geq 1-2\exp\left(-\frac{2\lambda^2}{NL_\theta^2 D^2}\right),
		\end{align*}
		Since this holds for a fixed $\txk$, a union bound over all $\{\txk\}_{k=1}^{N_\text{cover}}$ ensures that
		\begin{align*}
			\max_{k\in\{1,...,N_\text{cover}\}} \left|\frac{1}{N}\sum_{i=1}^N J(\txk, \thetai) - \E [J(\txk, \theta)]\right| \leq \frac{\lambda}{N},
		\end{align*}
		holds with probability at least
		$1-2N_\text{cover}\exp\left(-\frac{2\lambda^2}{NL_\theta^2 D^2}\right)$. Last, from Lemma 6.27 in \cite{mohri2018foundations}, we have
		$N_\text{cover}=\Ncal(r, \Xcal, \|\cdot\|_2)\leq\left(\frac{3R}{r}\right)^n$. Then the claim directly follows by plugging the obtained bounds in \eqref{proof:optimal_values:monotone}, and by setting $\beta = N_\text{cover}\exp\left(-\frac{2\lambda^2}{NL_\theta^2 D^2}\right)$.

	\subsection{Proof of Lemma \ref{lemma:bound_hdist} }
	Let $x_\text{PP}$ and $x_\text{RP}$ be minimizers, respectively, of PP$_{\veps-\delta}$ and RP$_{\veps+\delta}$. Since these problems minimize the same objective function, and since $\Fcal_{\text{PP}_{\veps-\delta}}\subseteq\Fcal_{\text{RP}_{\veps+\delta}}$, we always have $0 \leq J^\star_{\text{PP}_{\veps-\delta}} - J^\star_{\text{RP}_{\veps+\delta}},$
	which is indeed the first inequality in \eqref{lemma:optimality_hdist_claim0}. Now, we analyze the case in which $x_\text{PP}\in\Fcal_{\text{PP}_{\veps-\delta}}$ and $x_\text{RP}\in\Fcal_{\text{RP}_{\veps+\delta}}\setminus\Fcal_{\text{PP}_{\veps-\delta}}$ (otherwise $x_\text{PP}$ and $x_\text{RP}$ must necessarily lead to the same optimal value, i.e.\ $J^\star_{\text{PP}_{\veps-\delta}} = J^\star_{\text{RP}_{\veps+\delta}}$).
	Since $\Fcal_{\text{PP}_{\veps-\delta}}\subseteq\Fcal_{\text{RP}_{\veps+\delta}}$, we have
	$$\mathbb{d}(\Fcal_{\text{RP}_{\veps+\delta}}, \Fcal_{\text{PP}_{\veps-\delta}}) = \inf \left\{\rho \geq 0 : \Fcal_{\text{RP}_{\veps+\delta}} \subseteq \Fcal_{\text{PP}_{\veps-\delta}} \oplus \rho \B \right\},$$
	thus, similarly as Example 7.62 in \cite{rockafellar2009variational}, there exists $\bx\in\Fcal_{\text{PP}_{\veps-\delta}}$ whose distance to $x_\text{RP}$ is not greater than $\mathbb{d}(\Fcal_{\text{RP}_{\veps+\delta}}, \Fcal_{\text{PP}_{\veps-\delta}})$, i.e.:
	\begin{align}\label{proof:optimality_hdist:xbar}
		\exists \bx\in\Fcal_{\text{PP}_{\veps-\delta}}: \|\bx-x_\text{RP}\|_2 \leq \mathbb{d}(\Fcal_{\text{RP}_{\veps+\delta}}, \Fcal_{\text{PP}_{\veps-\delta}}).
	\end{align}
	Then, we have:
	\begin{align*}
		J^\star_{\text{PP}_{\veps-\delta}} - J^\star_{\text{RP}_{\veps+\delta}} &= \sum_{j=1}^K \hpj J(x_\text{PP}, \hthetaj) - \sum_{j=1}^K \hpj J(x_\text{RP}, \hthetaj) 
		\\& \leq \sum_{j=1}^K \hpj J(\bx, \hthetaj) - \sum_{j=1}^K \hpj J(x_\text{RP}, \hthetaj) 
		\\& \leq L_x \|\bx - x_\text{RP}\|_2 
		\\& \leq L_x \mathbb{d}(\Fcal_{\text{RP}_{\veps+\delta}}, \Fcal_{\text{PP}_{\veps-\delta}}),
	\end{align*}
	where the first inequality follows since $\bx$ is potentially suboptimal for PP$_{\veps-\delta}$, the second from the Lipschitz continuity of $J$, and the third one from \eqref{proof:optimality_hdist:xbar}. This proves the claim.
	
	\subsection{Proof of Lemma \ref{lemma:Hdist_polytopes} }
	Since we have $\Fcal_{\text{PP}_\veps} \subseteq \Fcal_{\text{RP}_\veps},$
	from \eqref{proof:hdist_bound:tFcal}-\eqref{proof:hdist_bound:tF_RP_k} we have
	\begin{align*}
		\d(\Fcal_{\text{PP}_\veps}, \Fcal_{\text{RP}_\veps})  
		&= \min \{\rho \geq 0: \cup_{k=1}^M \Fcal_{\text{RP}_\veps}^k \subseteq (\cup_{k=1}^M \Fcal_{\text{PP}_\veps}^k) \oplus \rho \B\}
		\\&= \min \{\rho \geq 0: \cup_{k=1}^M \Fcal_{\text{RP}_\veps}^k \subseteq \cup_{k=1}^M (\Fcal_{\text{PP}_\veps}^k \oplus \rho \B)\}
	\end{align*}
	where we leveraged that the Minkowski sum is distributive with respect to the union \cite{velichova2016notes}.
	Now, let 
	\begin{align*}
		\Bar{\rho} = \max_{k\in\{1,...,M\}} \d(\Fcal_{\text{PP}_\veps}^k, \Fcal_{\text{RP}_\veps}^k).
	\end{align*}
	Such $\Bar{\rho}$ satisfies
	$\Fcal_{\text{RP}_\veps}^k \subseteq (\Fcal_{\text{PP}_\veps}^k \oplus \Bar{\rho} \B), \forall k\in\{1,...,M\},$
	which implies
	$$\cup_{k=1}^M \Fcal_{\text{RP}_\veps}^k \subseteq \cup_{k=1}^M (\Fcal_{\text{PP}_\veps}^k \oplus \Bar{\rho} \B) = (\cup_{k=1}^M \Fcal_{\text{PP}_\veps}^k) \oplus \Bar{\rho} \B, $$
	thus, from the definition of Hausdorff distance:
	\begin{align}\label{proof:hdist_bound:dist_for_union}
		\d(\Fcal_{\text{PP}_\veps}, \Fcal_{\text{RP}_\veps}) \leq \Bar{\rho},
	\end{align}
	i.e., the Hausdorff distance between union of sets of the type described in \eqref{proof:hdist_bound:tF_PP_k}, \eqref{proof:hdist_bound:tF_RP_k}, and with an inclusion relation such as \eqref{proof:hdist_bound:tFcal}, is bounded by the largest distance between corresponding pairs constituting the unions in \eqref{proof:hdist_bound:tFcal}.
	
		Now, to bound $\d(\Fcal_{\text{PP}_\veps}^k, \Fcal_{\text{RP}_\veps}^k)$, we observe that $\Fcal_{\text{RP}_\veps}^k$ in \eqref{proof:hdist_bound:tF_RP_k} is of the form  $A^kx\leq b^k+\gamma^k$, with $A^k$ obtained by stacking vertically the matrices $C_h, \forall h\in\Hcal_k$, $b^k$ obtained by stacking vertically $-D_h\hthetaj - b_h$, and $\gamma^k$ obtained by stacking vertically $\gamma^{(j)}_h$, $\forall h\in\Hcal_k, \forall j\in\Jcal_k$. Similarly, $\Fcal_{\text{PP}_\veps}^k$ is of the form  $A^kx\leq b^k - \tau^k$, for an appropriate $\tau^k$ derived with the same reasoning. Hence, we have
		\begin{align*}
			&\Fcal_{\text{RP}_\veps}^k = \{x\in\R^n: A^kx\leq b^k + \gamma^k\},
			\\& \Fcal_{\text{PP}_\veps}^k = \{x\in\R^n: A^kx\leq b^k - \tau^k\},
		\end{align*}
		i.e., $\Fcal_{\text{PP}_\veps}^k$ and $\Fcal_{\text{RP}_\veps}^k$, are polytopic sets that differ only in a tightening or relaxation parameter, $\forall k\in\{1,...,M\}$.
		Let $A^k_{[i]}$ be any invertible square submatrix of $A^k$ of dimension $n$, and let $\sigma^k_{\text{min}, i}$ be its smallest singular value, $\forall i\in\{1,...,n^k_\text{s}\}$, where $n^k_\text{s}$ is the number of the square invertible submatrices of $A^k$. Let $\sigma(A^k) = \min_{i\in\{1,...,n_\text{s}^k\}} \sigma^k_{\text{min}, i}$. 
		From \cite{conforti2014integer} (Theorem 3.34), $v_{\text{RP}, i}^{k}$ is a vertex of $\Fcal_{\text{RP}_\veps}^k$ if and only if the inequalities $A^k v_{\text{RP}, i}^{k} \leq b^k + \gamma^k$ hold with equality for $n$ linearly independent rows of $A^k$. A similar argument holds for $\Fcal_{\text{PP}_\veps}^k$. Hence, the vertices of $\Fcal_{\text{RP}_\veps}^k$ and $\Fcal_{\text{PP}_\veps}^k$ can be determined from
		\begin{align*}
			v_{\text{RP}, i}^{k} = (A_{[i]}^{k})^{-1} (b^k_{[i]} + \gamma^k), \quad  v_{\text{PP}, i}^{k} = (A_{[i]}^{k})^{-1} (b^k_{[i]} - \tau^k)
		\end{align*}
		for $i\in\{1,...,n^k_\text{s}\}$,
		which in particular tells us that the vertices of $\Fcal_{\text{PP}_\veps}^k$ are given by the intersection of the corresponding translated faces of $\Fcal_{\text{RP}_\veps}^k$.
		Now, it is easy to observe that
		\begin{align}\label{proof:vertex_1}
			\d(\Fcal_{\text{RP}_\veps}^k, \Fcal_{\text{PP}_\veps}^k) = \max_{i\in\{1,...,n^k_\text{s}\}} \|v_{\text{RP}, i}^{k} - v_{\text{PP}, i}^{k}\|_2,
		\end{align}
		since from \cite{wills2007hausdorff} we know that the Hausdorff distance between bounded convex sets is the Hausdorff distance between their boundaries, and since the largest distance between two parallel faces in $\Fcal_{\text{RP}_\veps}^k, \Fcal_{\text{PP}_\veps}^k$ is attained at the corresponding vertices. By substituting the expressions for $v_{\text{RP}, i}^{k}, v_{\text{PP}, i}^{k}$, we have
		\begin{align*}
			\|v_{\text{RP}, i}^{k} - v_{\text{PP}, i}^{k}\|_2 
			= \|(A_{[i]}^{k})^{-1} (\gamma^k - \tau^k) \|_2
			\leq \frac{\|\gamma^k - \tau^k\|_2}{\sigma(A^k)}  
		\end{align*}
		where we use the known fact that $\|(A_{[i]}^{k})^{-1}\|_2 = \frac{1}{\sigma(A^k)}$ (see \cite{demmel1997applied}, Theorem 3.3). Thus, from \eqref{proof:vertex_1}, we have
		\begin{align*}
			\d(\Fcal_{\text{RP}_\veps}^k, \Fcal_{\text{PP}_\veps}^k) \leq \max_{h\in\mathcal{H}_k}\max_{j\in\mathcal{J}_k} \frac{\|\gamma_{h}^{(j)} - \tau_{h}^{(j)}\|_2}{\sigma(C_h)},
		\end{align*}
		and from \eqref{proof:hdist_bound:dist_for_union} we finally obtain the claim.

	$\hfill\blacksquare$
	
	\subsection{Proof of Lemma \ref{lemma:optimality_hdist} }
	Let $\Pcal = \Biggl\{P_q: P_q = \sum_{j\in\Jcal_q} \hpj, \Jcal_q\subset\{1,...,K\}, q\in\{1,...,2^K\}\Biggr\}$
	i.e., $\Pcal$ contains all the terms that can be expressed as sum of elements from $\{\hat{p}^{(1)}, ..., \hat{p}^{(K)}\}$ (thus, $\Pcal$ contains at most $2^K$ distinct elements), which we assume to be ordered in increasing order. This also means that the left-hand side of the chance constraint in  $\Fcal_{\text{RP}_{\veps}}$ in \eqref{eq:feas_set_rp} is one of the elements of $\Pcal$ $\forall x\in\Fcal_{\text{RP}_{\veps}}$, which also implies that $\Fcal_{\text{RP}_{\veps}}$ does not change for $\veps\in[P_q, P_{q+1})$.
	
	Now, in view of \eqref{lemma:feas_set_jump:assum}, we have
	$P_q < \veps < P_{q+1},$
	for some $q\in\{1,...,|\Pcal|-1\}$, which means that there always exists a $\delta$, positive and sufficiently small, such that
	$P_q < \veps - \delta < P_{q+1}$ and $ P_q < \veps + \delta < P_{q+1},$
	hold true simultaneously, or equivalently:
	\begin{align*}
		&\delta < \min \{\veps - P_q, P_{q+1}-\veps\}
		\\&  \Leftrightarrow \delta \in\Biggl(0, \ \min_{\Jcal\subset\{1,...,K\}} \biggl|\veps - \sum_{j\in\Jcal} \hpj \biggr|\Biggr).
	\end{align*}
	Thus, we have proved that for any $\delta$ that satisfies \eqref{lemma:feas_set_jump:delta_interval}, it holds that
	$\Fcal_{\text{RP}_{\veps-\delta}} = \Fcal_{\text{RP}_{\veps+\delta}}.$
	Finally, from Lemma \ref{lemma:bound_hdist} and \ref{lemma:Hdist_polytopes}, we have
	\begin{align*}
		J^\star_{\text{PP}_{\veps-\delta}} - J^\star_{\text{RP}_{\veps+\delta}} 
		& \leq L_x \mathbb{d}(\Fcal_{\text{RP}_{\veps-\delta}}, \Fcal_{\text{PP}_{\veps-\delta}}) 
		\\& \leq L_x \max_{\substack{j\in\{1,...,K\} \\ h\in\{1,...,Z\}}} \frac{\|\tau^{(j)}_{h}-\gammaj_h\|_2}{\sigma(C_h)} ,
	\end{align*}
	and the claim follows from Theorem \ref{th:optimal_values} .

	\section*{Acknowledgment}
	We thank Peyman Mohajerin Esfahani and Remy Spliet for the fruitful discussions.

	\section*{References}
	\bibliographystyle{IEEEtran}
	
	\bibliography{bibliography.bib}

	\begin{IEEEbiography}[{\includegraphics[width=1in,height=1.25in,clip,keepaspectratio]{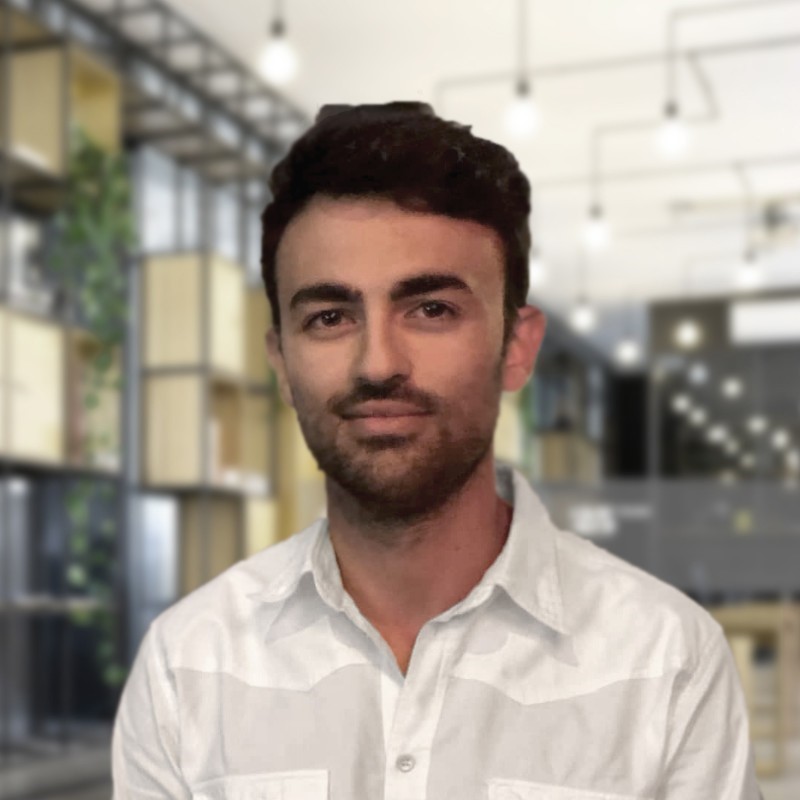}}]{Francesco Cordiano} received the B.Sc. degree from Politecnico di Milano, Italy, and the M.Sc. degree from ETH Zurich, Switzerland. He is currently a Ph.D. candidate at the Delft Center for Systems and Control, Delft University of Technology, The Netherlands.
		
		His current research interests include stochastic optimization, reinforcement learning, and model predictive control of hybrid systems.
		
	\end{IEEEbiography}
	
	\begin{IEEEbiography}[{\includegraphics[width=1in,height=1.25in,clip,keepaspectratio]{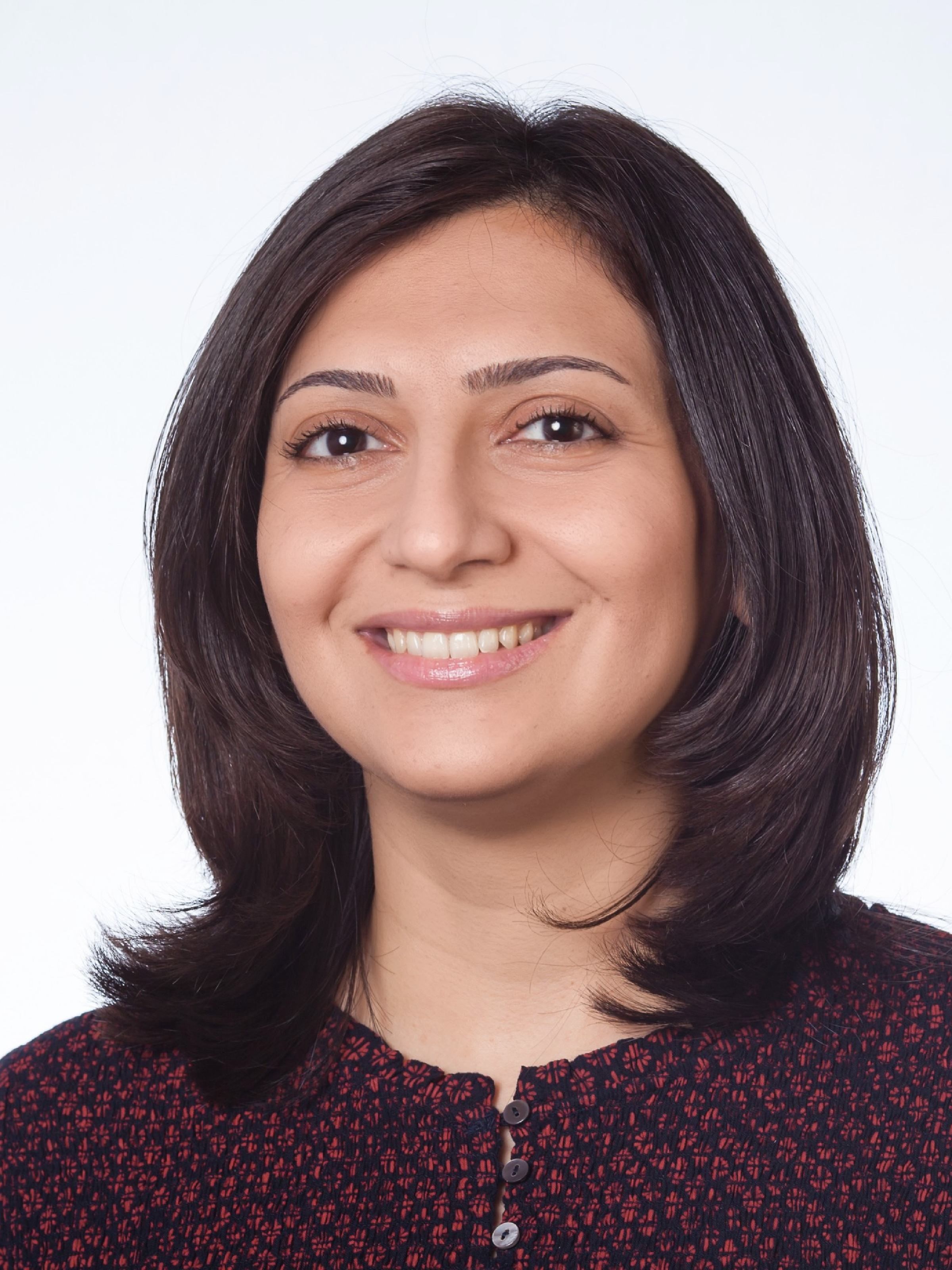}}]{Matin Jafarian} (Member, IEEE) is currently an Assistant Professor at Delft University of Technology (TU Delft), The Netherlands. Her theoretical research interests include modeling, analysis and control of nonlinear, and stochastic dynamic networks.
	\end{IEEEbiography}
	
	\begin{IEEEbiography}[{\includegraphics[width=1in,height=1.25in,clip,keepaspectratio]{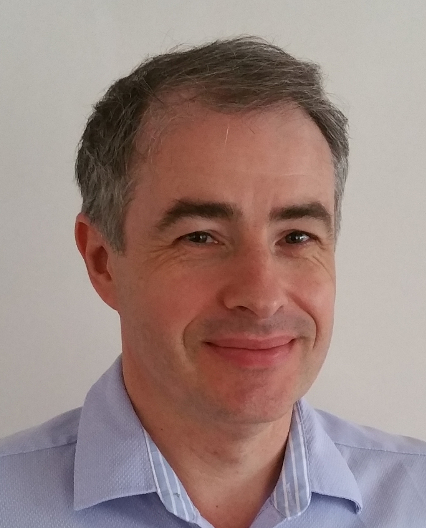}}]{Bart De Schutter} (IEEE member since 2008, senior member since 2010, fellow since 2019) is a full professor and head of department at the Delft Center for Systems and Control of Delft University of Technology in Delft, The Netherlands. 
		
	Bart De Schutter is senior editor of the IEEE Transactions on Intelligent Transportation Systems.  His current research interests include integrated learning- and optimization-based control, multi-level and multi-agent control, and control of hybrid systems.      
	\end{IEEEbiography}
	
\end{document}